\documentclass[12pt]{amsart}
\usepackage{amsmath}
\usepackage{mathrsfs}
\usepackage{mathabx}
\usepackage{enumerate}
\usepackage{skak}
\usepackage{bbm}
\usepackage{skull}
\usepackage{mathrsfs}
\usepackage{amsmath,amssymb,amsfonts}
\usepackage[all,cmtip]{xy}
\usepackage{hyperref}
\usepackage{a4wide}
\usepackage[mathscr]{eucal}
\usepackage{paralist}

\usepackage [english]{babel}
\usepackage [autostyle, english = american]{csquotes}
\MakeOuterQuote{"}

\theoremstyle{plain}
\newtheorem{thm}{Theorem}[section]
\newtheorem{lem}[thm]{Lemma}
\newtheorem{cor}[thm]{Corollary}
\newtheorem{prop}[thm]{Proposition}

\newtheorem{conj}[thm]{Conjecture}
\newtheorem{stat}[thm]{Statement}

\theoremstyle{definition}

\newtheorem*{fact}{Fact}

\newtheorem{defn}{Definition} 
\newtheorem{ex}[thm]{Exemple}

\newtheorem{rmk}[thm]{Remark}

\numberwithin{thm}{section}

\newcommand{\Ess}{{\rm EFin}}
\newcommand{\sect}{{\rm Section}}

\newcommand{\Sch}{{\rm Sch}}
\newcommand{\Hom}{{\rm Hom}}
\newcommand{\Res}{{\rm Res}}

\newcommand{\PI}{\rm PI}
\newcommand{\Spec}{{\rm Spec \,}}

\newcommand{\Aff}{{\rm Aff}}

\newcommand{\Gal}{{\rm Gal}}

\newcommand{\Aut}{{\rm Aut}}

\newcommand{\A}{{\mathbb A}}

\newcommand{\C}{{\mathbb C}}

\newcommand{\F}{{\mathbb F}}
\newcommand{\G}{{\mathbb G}}

\newcommand{\N}{{\mathbb N}}
\renewcommand{\P}{{\mathbb P}}
\newcommand{\Q}{{\mathbb Q}}
\newcommand{\R}{{\mathbb R}}

\newcommand{\Y}{{\mathbb Y}}
\newcommand{\Z}{{\mathbb Z}}

\newcommand{\Rep}{\text{\sf Rep}}
\newcommand{\id}{{\rm id\hspace{.1ex}}}

\begin{document}
\title{Nori's Fundamental Group over a  non-Algebraically Closed Field}

\author{ Lei Zhang }
 \address{
    Freie Universit\"at Berlin\\
    FB Mathematik und Informatik\\
    Arnimallee 3, Zimmer 112A
    14195 Berlin, Deutschland }
\email{l.zhang@fu-berlin.de}
\thanks{This work was supported by the European Research Council (ERC) Advanced Grant 0419744101 and the Einstein Foundation}
\date{\today}

\begin{abstract} Let $X$ be a connected reduced scheme over a field $k$,  $x\in X(k)$ be a $k$-rational point. M. V. Nori constructed in his Ph.D thesis a fundamental
group scheme $\pi^N(X,x)$ which generalizes A. Grothendieck's \'etale fundamental group $\pi_1^{\text{\'et}}(X,x)$ by including infinitesimal coverings.  However, Nori's fundamental group scheme carries little arithmetic information, and it behalves like the \'etale fundamental group only when  $k$ is algebraically closed. For example, if $X=\Spec(k)$, then Nori's fundamental
group scheme is always trivial while the \'etale
fundamental group
$\pi_1^{\rm\text{\'et}}(X,x)=\Gal(\bar{k}/k)$. In this paper, we study a slightly modified version of Nori's fundamental group scheme: We take $x$ to be a geometric point instead of a rational point. It is very suprising to the author that this tiny little modification of Nori's original definition brings a lot of arithmetic information and makes the fundamental group scheme more like $\pi_1^{\text{\'et}}(X,x)$. For example, now if we take $X=\Spec(k)$ again, with $\bar{x}\in X(\bar{k})$, then we get a profinite group scheme $\pi^N(k/k,\bar{x})$ over $k$ which takes $\Gal(\bar{k}/k)$ as its (pro-constant) quotient. Thus not only the Galois extensions, but also the purely inseparable extensions of $k$ are encoded into $\pi^N(k/k,\bar{x})$. We call $\pi^N(k/k,\bar{x})$ the \textit{Nori-Galois group} of $k$. We also studied the fundamental sequence which relates the Nori-Galois group to the geometric fundamental group. It turns out that the expected fundamental exact sequence is always a complex and exact on the right, but fails to be exact in the middle and on the left. Then we give conditions to determine when the exactness holds.
\end{abstract} 
\setcounter{section}{-1}
\maketitle
\section{Introduction}
Let  $X$ be a  connected scheme, $x\in X(\bar{k})
$ be a geometric point. Let ECov$(X)$ be the category of finite \'etale coverings of $X$. Then we have a fibre functor $F$ from ECov($X$) to the category of finite sets by sending any finite \'etale  covering $f: Y\to X$ to its fibres $f^{-1}(x)$. In \cite{SGA1}[Expos\'e V] A. Grothendieck proved that ECov($X$) together with the fibre functor $F$ forms a Galois category. Then he defined the \'etale fundamental group $\pi_1^{\text{\'et}}(X,x):=\Aut(F)$ to be the group of automorphisms of $F$. This is  a {profinite group},
i.e.  a topological group of the form $\varprojlim_{i\in I}G_i$ where $I$ is a small cofiltered category and $G_i$ is a finite group for each $i\in I$. This profinite group classifies all torsors under finite   groups, in particular when $G$ is a finite abelian group we have $H_{\text{\'et}}^1(X,G)=\Hom_{\text{cont}}(\pi_1^{\text{\'et}}(X,x),G)$. If $X=\Spec(k)$ and $x\in X(\bar{k})$ corresponds to the field extension $k\subseteq \bar{k}$, then $I$ can be chosen as the set of finite Galois sub-extensions of $k\subseteq \bar{k}$, and for each  $i=(k\subseteq K)\in I$, $G_i:=\Gal(K/k)$, so $\pi_1^{\text{\'et}}(X,x)=\varprojlim_{i\in I}G_i$ is just
the absolute Galois group $\Gal(\bar{k}/k)$.

Let $X$ be a proper reduced connected  scheme over a 
field $k$, $x\in X(k)$ be a rational point. In \cite[Part I,
Chapter I]{Nori} M. V. Nori constructed a full subcategory $\Ess(X)\subseteq {\rm Vec}(X)$ of the category of vector bundles on $X$. Objects in $\Ess(X)$ are called essentially finite vector bundles. He proved that $\Ess(X)$ with the fibre functor $\omega$ from $\Ess(X)$ to the category of finite dimensional vector spaces sending $V\mapsto V|_x$ is a Tannakian category over $k$. Then he defined $\pi^N(X,x):=\Aut^{\otimes}(\omega)$ to be the group of $k$-linear tensor automorphisms of $\omega$. This is a profinite $k$-group scheme which classifies all $k$-pointed torsors over $X$ under finite $k$-group schemes.  In particular when $G$ is a finite abelian $k$-group scheme we have $H_{\text{fppf}}^1(X,G)=\Hom_{\text{grp.sch}}(\pi_1^{N}(X,x),G)$. However, in this construction the properness assumption is vital, it does not apply to non-proper schemes. To remedy this M. V. Nori introduced in \cite[Part I,
Chapter II]{Nori} another construction. For any  reduced connected  scheme $X$ over a
field $k$ with a rational point $x\in X(k)$, let $N(X/k,x)$ be the category of torsors under finite $k-$group schemes with a fixed
 $k$-point lying over $x$. Then 
Nori defined $\pi^N(X,x):=\varprojlim_{i\in N(X/k,x)}G_i$, where $G_i$ is the finite group scheme corresponding to the pointed torsor $i$. It is not hard to prove that this definition concides with the Tannakian one (See \cite[Part I,
Chapter I, Proposition 3.11]{Nori} and \ref{unicov} for an explaination).

There is a comparison between Grothendieck's fundamental group and Nori's fundamental group scheme: If $X$ is a connected reduced scheme over an algebraically closed field $k$ with a rational point $x\in X(k)$, then $\pi^N(X,x)(k)\cong \pi_1^{\text{\'et}}(X,x)$ as topological groups, where $\pi^N(X,x)(k)$ is equipped with the Zariski topology. This means, \textit{over an algebraically closed field}, Nori's fundamental group scheme is a generalization of Grothendieck's \'etale fundamental group. However, when the base field is not algebraically closed, Nori's definition is quite different from Grothendieck's: 

\begin{enumerate}\item If $X=\Spec(k)$, then Nori's fundamental
group scheme is always trivial while the \'etale
fundamental group
$\pi_1^{\rm\text{\'et}}(X,x)=\Gal(\bar{k}/k)$.
\item In \cite[Part I,
Chapter II, Proposition 5]{Nori} Nori proved that the fundamental group scheme satisfies base change by separable field extensions. But this does hold for the \'etale fundamental group.  Take the projective space for example, if we see $\pi_1^{\text{\'et}}(\P_{{\Q}}^n)$ as a profinite group scheme over $\Q$, then we have $\pi_1^{\text{\'et}}(\P_{{\Q}}^n)\times_{\Q}\bar{\Q}=\Gal(\bar{\Q}/\Q)\times_{\Q}\bar{\Q}\neq \{1\}=\pi_1^{\text{\'et}}(\P_{\bar{\Q}})$.
\item  By \cite[Chapter II, Proposition 4]{Nori}, $\pi^N(X,x_1)$ and $\pi^N(X,x_2)$ differ by an inner twist for different rational points $x_1,x_2\in X(k)$, and they only become isomorphic after base change to $\bar{k}$.
\end{enumerate}

The first two properties reveals that Nori's fundamental group scheme is in some sense a \textit{geometric fundamental group}, i.e. it is better designed  for schemes over an algebraically closed field.  In this paper, we are going to  study an arithematic variant of this fundamental group scheme which brings it closer to $\pi_1^{\text{\'et}}$. In Nori's second definition of the fundamental group scheme, instead of taking a rational point $x\in X(k)$, we take a geometric point $\bar{x}\in X(\bar{k})$. It is really surprising that this tiny little modification makes the fundamental group scheme contain extremely rich arithmetic information.

To simplify the study we first split the fundamental group scheme into several different parts and study each of them. Let $N(X/k,\bar{x})$ be as before, except that $x$ is now replaced by a geometric point $\bar{x}$, and denote $\pi^N(X/k,\bar{x})$ the group scheme $\varprojlim_{i\in N(X/k,\bar{x})}G_i$. Let $I_{\text{\'et}}(X/k,\bar{x})$ (resp. $I_{co}(X/k,\bar{x}), I_{lc}(X/k,\bar{x})$) be the full subcategory of $N(X/k,\bar{x})$
consisting of those pointed torsors whose group schemes are \'etale (resp. constant, local), and the corresponding fundamental
group is denoted by $\pi^E(X/k,\bar{x})$ (resp. $\pi^G(X/k,\bar{x})$, $\pi^L(X/k,\bar{x})$). Then according to \ref{comparison}, we have the following  canonical surjections:
     $$\xymatrix{&\pi^E(X/k,\bar{x})\ar@{->>}[r]&\pi^G(X/k,\bar{x})\\\pi^N(X/k,\bar{x})\ar@{->>}[ur]\ar@{->>}[dr]&&\\&\pi^L(X/k,\bar{x})&} $$
$$\xymatrix{\pi^N(X/k,\bar{x})\ar@{>>}[rr]&& \pi^E(X/k,\bar{x})\times_k\pi^L(X/k,\bar{x})}.$$
 In fact, $\pi^G(X/k,\bar{x})$ is nothing but a "group
    scheme version" of $\pi_1^{\text{\'et}}(X,\bar{x})$ (see \ref{constant2} (ii)).   Although $\pi^E(X/k,\bar{x})$ and $\pi_1^{\text{\'et}}(X,\bar{x})$ are all fundamental groups classifying {\'etale coverings}
 they are indeed largely different. For example  $\pi_{1}^{\text{\'et}}(\Spec(\R),\bar{x})=\Gal(\C/\R)=\Z/2\Z$ and the universal covering of $\Spec(\R)$ under $\pi_{1}^{\text{\'et}}(\Spec(\R),\bar{x})$ is $\Spec(\C)$, while we have

\begin{thm} (See \ref{R}). Let $\R$ be the field of real numbers, $\bar{x}:\Spec(\C)\to \Spec(\R)$ be the morphism corresponding to the natural inclusion $\R\subset \C$. Then $$\pi^N(\R/\R,\bar{x})=\pi^E(\R/\R,\bar{x})=\varprojlim_{n\in\N^+}\mu_{n,\R}$$ is an  infinite $\R-$group scheme, and the universal covering corresponding to  $\pi^E(\R/\R,\bar{x})$ is a non-Noetherian affine scheme with infinitely many connected components.
\end{thm}

Although the so defined fundamental group scheme is quite complicated, it does behalve well as an arithmetic fundamental group scheme:
\begin{prop}\label{field0} (See \ref{field}). Let $X=\Spec(k)$ be a field, $\bar{x}\in X(\bar{k})$ be the geometric point $k\subseteq \bar{k}$. Then:
\begin{enumerate}[(i)]
\item $\pi^L(k/k,\bar{x})=\{1\},\ \pi^N(k/k,\bar{x})=\pi^E(k/k,\bar{x})$,\\ when $k$ is a perfect field;
\item $\pi^E(k/k,\bar{x})=\{1\},\ \pi^N(k/k,\bar{x})=\pi^L(k/k,\bar{x})$,\\ when $k$ is a separably closed field;
\item $\pi^N(X/k,\bar{x})=\pi^E(X/k,\bar{x})=\pi^L(X/k,\bar{x})=\{1\}$,\\  when $X$ is $\A_{\bar{k}}^n$ with $k$ a field of characteristic 0 or $X$ is $\P_{\bar{k}}^n$ with $k$ a field of arbitrary characteristic.
\end{enumerate} \end{prop}
\noindent In \cite{Ro} you can find an interesting application of \ref{field0}.

The following is an analogue of \cite[Expos\'e X, Corollarie 1.8., pp. 204]{SGA1}:

\begin{prop} (See \ref{basechange}). Let $X$ be a scheme geometrically connected proper separable
over a field $k$, $k\subseteq l\subseteq l'$ be a sequence of field
extensions, where $l$ and $l'$ are algebraically closed fields. Let
$\bar{x}:\Spec(l')\to X$ be a geometric point. Then the following
natural map
$$\pi_{l}^{l'}:\pi^E(X\times_kl'/k,\bar{x})\longrightarrow\pi^E(X\times_kl/k,\bar{x})$$ is an isomorphism of $k$-group schemes.
\end{prop}

\noindent By contrast $\pi^L(X/k,\bar{x})$ (hence also $\pi^N(X/k,\bar{x})$) doesn't satisfy base change by algebraically closed field extensions. This can be deduced from a famous counterexample by Mehta and Subramanian in \cite{MS} which was used to show that base change by algebraically closed field extensions fails for Nori's original definition (See \ref{basechangeinf} for details).

The following theorem shows that the arithmetic fundamental group scheme we are considering here deserves the name \textit{fundamental group}.

\begin{prop} (See \ref{basepoint}).  Let $X$ be any connected reduced scheme over $k$, $\bar{x}_1: \Spec(\bar{l}_1)\to X$ and $\bar{x}_2: \Spec(\bar{l}_2)\to X$ be two  geometric points of $X$. Then there are (non-canonical) isomorphisms between the following $k$-group schemes: \[\pi^{E}({X}/k,\bar{x}_1)\cong \pi^{E}({X}/k,\bar{x}_2)\tag{i}\] \[\pi^{L}({X}/k,\bar{x}_1)\cong \pi^{L}({X}/k,\bar{x}_2)\tag{ii}\]\[\pi^{N}({X}/k,\bar{x}_1)\cong \pi^{N}({X}/k,\bar{x}_2).\tag{iii}\]
\end{prop}

A very powerful tool to understand   arithmetic fundamental groups is the so called \textit{fundamental exact sequence} which relates the geometric part (the geometric fundamental group) to the arithmetic part (the Galois group). 
Unlike the fundamental exact sequence for $\pi_1^{\text{\'et}}$ in \cite[Expos\'e IX, Th\'eor\`eme 6.1]{SGA1}, ours is exact only in certain cases. Nonetheless it does provide some valuable information about the arithmetic fundamental group, e.g. it follows immediately from the following theorem that our arithmetic fundamental group will \textit{never} satisfy base change by separable field extensions if $k\neq k^{\text{sep}}$.

\begin{thm} (See \S 3). Let $X$ be a geometrically connected scheme which is separable  over a field $k$ and let $\bar{x}\in X(\bar{k})$ be a geometric point, then there is a complex of $k$-group schemes \begin{equation}\label{seq1}1\to\pi^{I}(\bar{X}/k,\bar{x})\to \pi^{I}(X/k,\bar{x})\to  \pi^{I}(k/k,\bar{x})\to 1\tag{1}\end{equation}where $I=N,L$ or $E$. This sequence  is always exact on the right, not always exact on the left (\ref{localcase}), and   is  exact in the middle  if and only if for any object $(P,G,p)\in I(X/k,\bar{x})$ both of the following conditions are satisfied: \begin{enumerate}[(i)]\item If $(P,G,p)$ is saturated, then the image of the composition of the natural homomorphisms $$\pi^{I}(\bar{X}/k,\bar{x})\to \pi^{I}(X/k,\bar{x})\twoheadrightarrow G$$ is a normal subgroup of $G$; \item Whenever the pull-back of $(P,G,p)$ along $\bar{X}\to X$ is trivial there is an object $(Q,H,q)\in I(k/k,\bar{x})$ whose pull-back along $X\to \Spec(k)$ is isomorphic  to $(P,G,p)$.
\end{enumerate}

\noindent Moreover, condition (i) holds for triples $(P,G,p)$ where $G$ is \'etale and $P$ is connected (\ref{topthm}) or $G$ is local and $k$ is perfect (\ref{localcase}). But as the example (\ref{normality}) shows, (i) fails when $P$ is not connected while $G$ is still \'etale. The condition  (ii) holds when $k$ is perfect or $X$ is proper or $G$ is \'etale (\ref{descentprop}), but fails (\ref{faildescent}) when $k$ is not perfect, $X$ is not proper and $G$ is not \'etale.
\end{thm}

\noindent Here, I think  the state of the art is the counterexample \ref{normality}. In fact the normality problem as in (i) is always  very difficult  in the   study of the homotopy sequence of fundamental groups.  For example, in \cite{Zh2} an essential part of the proof is devoted to the normality problem, and it is also the key issue in \cite{EHV}. Here we give a very delicate example to show that the normality condition breaks in many cases. But of course, if one only restricts to the abelian quotient $\pi_{ab}^N$ of the whole fundamental group scheme $\pi^N$ then one gets an exact sequence as long as condition (ii) holds. 

The following is another special case in which (i) and (ii) hold.


\begin{cor} (See \ref{AP}).  If either $X=\A_{{k}}^n$ with $k$ is a field of characteristic 0 \textit{or} $X$ is a complete rational variety over an arbitrary field $k$, then the  canonical map $$\pi^N({X}/k,\bar{x})\to\pi^N(k/k,\bar{x})$$ is an isomorphism. In other words, any finite flat torsor over $X$ descends uniquely to $k$.
\end{cor}

As we will see in \ref{geometricfundamentalgroup}, there are two geometric fundamental group schemes corresponding to this arithmetic fundamental group scheme. Here is the fundamental sequence with respect to the other geometric fundamental group scheme.

\begin{thm} (See \S 4). Let $X$ be a scheme geometrically connected separable  over a  field $k$, $\bar{x}\in X(\bar{k})$ be a geometric point, then there is a
 natural sequence of $\bar{k}$-group schemes \begin{equation}\label{seq2}1\to\pi^{I}(\bar{X}/\bar{k},\bar{x})\to \pi^{I}(X/k,\bar{x})\times_k\bar{k}\to  \pi^{I}(k/k,\bar{x})\times_k\bar{k}\to 1.\tag{2}\end{equation} It is a complex,  always exact on the right, exact on the left when $k$ is perfect and $X$ is q.s. and q.c., but it is in general not  exact in the middle for $I=N,E,L$.
\end{thm}

In the end we apply our discussions to construct a possibly smaller subset $\sect_{\sim}^N(k,X)$ of the full set of section classes of the fundamental exact sequence of the \'etale fundamental group (see pp. \ref{Nori section}). In fact this subset contains all the "geometric sections", i.e. those sections which come from the rational points of $X$. Thus if one expect that there is a one-to-one correspondence between the rational points and the section classes, then a priori one should expect a one-to-one correspondence between the rational points and $\sect_{\sim}^N(k,X)$. Therefore, we formulate this possibly weaker version of the section conjecture (\ref{reform conj}). Indeed this formulation has an advantage when one deals with the problem in characteristic $p>0$ (see \ref{char p}).

{\bf Acknowledgments:}  I would like to express my deepest gratitude
to my Ph.D advisor H\'el\`ene Esnault for leading me into this
beautiful $\pi_1$ world, supporting my career and influencing me
from her work. I thank my friend Jilong Tong for a lot of extremely constructive
discussions. I also thank M. Romagny,  A. Vistoli, G. Zalamansky for their interest and helpful discussions, and  O. Wittenberg for very useful comments on the last section of this paper.

\section*{Notations and Conventions}
\begin{enumerate}[(i)]\item We always use $k$ to denote a field, $\bar{k}$ to denote its algebraic  closure.

\item Let $f: S'\to S$ be a morphism of schemes, $X'$ be a scheme over $S'$. We say $X'$ possess an $S$-\textit{form} if there is a scheme $X$ over $S$ whose pull-back along $f$ is isomorphic to $X'$. 
\item When $X$ is a scheme over $k$, we use $\bar{X}$ to denote $X\times_k\bar{k}$. If $k\subseteq K$ is a field extension we use $X_K$ to denote $X\times_kK$. Sometimes we also use $\bar{X}$ (resp.$X_K$) to denote something over $\bar{k}$ (resp. $K$) which does not necessarily possess a $k$-form $X$. This depends on the situation we are in.
\item Let $X\times_SY$ be a fibred product of schemes. We use $pr_1$ to denote the first projection $X\times_SY\to X$ and $pr_2$ to denote the second projection. 
\item Let $G$ be a group scheme over $k$. In this note, a $G$-\textit{torsor}  over a $k$-scheme $X$  is an $X$-scheme $P$ equipped with a right action $\rho: P\times_kG\to P$, where $\rho$ is a morphism of $X$-schemes which induces an isomorphism $pr_1\times\rho: P\times_kG\to P\times_XP$. Moreover we require that the structure map $P\to X$ of the $X$-scheme $P$ is faithfully flat and quasi-compact.
\item Let $X$ be a scheme. We use $X_{\rm red}$ to denote the reduced closed subscheme structure of $X$.
\item Let $f:X\to S$ be a morphism of schemes. We call $f$ \textit{separable} \cite[Expos\'e X, D\'efinition 1.1]{SGA1} if it is flat and  all its   geometric fibres  are  reduced. A $k$-scheme $X$ is called separable if and only if its structure map is separable.
\item Let $G$ be a group scheme over $k$, $H\subseteq G$ be a subgroup scheme. We say $H\subseteq G$ is a \textit{normal subgroup scheme} if for any $k$-scheme $T$, $H(T)\subseteq G(T)$ is a normal subgroup. Note that $H\subseteq G$ is normal if and only if $\bar{H}\subseteq \bar{G}$ is normal.
\item Let $f:X\to Y$ be a morphism of schemes. We say $f$ is \textit{surjective} if for each morphism $y: \Spec(k)\to Y$ the fibred product of $y$ with $f$ is a non-empty scheme. If $f: X\to Y$ is a morphism of group schemes, then we say $f$ is \textit{surjective}  when $f$ is surjective as a morphism of FPQC-sheaves of groups. In the case when $X,Y$ are affine group schemes over $k$, $f$ is surjective $\Leftrightarrow$ the corresponding map of Hopf-algebras is injective.
\item Let $G$ be a $k-$group scheme. We denote by  $\Rep_k(G) $ the category of finite dimensional $k$-linear representations of $G$.
\item Let $S'\to S$ be a Galois covering, i.e. a connected finite \'etale covering which is a torsor under its own automorphism group $\Aut_S(S')$. Let $\pi': X'\to S'$ be a morphism of schemes. A \textit{twisted action} of $\Aut_S(S')$ on $X'$ is a group homomorphism $f:\Aut_S(S')\to\Aut(X')$, where $\Aut(X')$ is the group of scheme automorphisms of $X'$, such that for any $\sigma\in\Aut_S(S')$ the following diagram $$\xymatrix{X'\ar[r]^-{f(\sigma)}\ar[d]_-{\pi'}& X'\ar[d]^-{\pi'}\\S'\ar[r]^-{\sigma}&S'}$$ is commutative. By Grothendieck's general descent theory \cite[6.2, Example B, pp. 139]{BLR}, there is an equivalence of categories between the category of affine $S'$-schemes equipped with a twisted action from $\Aut_S(S')$ and the category of affine $S$-schemes. We often refer to this as \textit{Galois descent}.
\item Here is another version of \textit{Galois descent}. Let $k\subseteq K$  be a finite Galois extension. There is an equivalence of categories between the category of finite abstract groups equipped with a continous action from $\Gal(\bar{k}/k)$ (resp. an action from $\Gal(K/k)$) via {group automorphisms} and the category of finite \'etale $k$-group schemes (resp.  $k$-group schemes whose pull-back to $K$ are finite constant). \cite[3.25-3.26]{AV}.
\end{enumerate}

\maketitle
\section{The Arithmetic Nori's Fundamental Group}

Let $X$ be a reduced connected  scheme over a
field $k$, $x:S\to X$ be a  morphism of $k$-schemes with $S$ non-empty.

\begin{defn}\label{torsor} Consider the triples $(P,G,p)$ where $G$ is a finite group  scheme over $k$, $P$ is a  $G$-torsor over $X$, $p:S\to P$ is a $k$-morphism lifting $x:S\to X$. A morphism from $(P_1,G_1,p_1)$ to $(P_2,G_2,p_2)$ is a pair $(s,t)$ where $t:G_1\to G_2$ is a $k$-group scheme homomorphism, $s:P_1\to P_2$ is an $X$-scheme morphism which intertwines the group action and sends $p_1\mapsto p_2$.  We denote the category consisting of such triples by $N(X/k,x)$.
\end{defn}


\begin{defn}\label{cofilter} \cite[Expos\'e I, D\'efinition 2.7]{SGA4} A   category $I$ is called cofiltered if it satisfies the following three conditions:\begin{enumerate}[(i)]
\item it is non-empty;
\item for any objects $i,j\in I$, there exists an object $k\in I$ and two arrows $k\to i$, $k\to j$;
\item for any two morphisms $$\xymatrix{j\ar@/^/[r]^-a\ar@/_/[r]_-b &i}$$ there exists a morphism $c:k\to j$ satisfying $a\circ c=b\circ c$.
\end{enumerate}
\end{defn}

\begin{rmk} The category $N(X/k,x)$ has finite fibred products and  a final object $(X,\{1\},x)$, so in particular it is cofiltered. The proof  is due to M.V.Nori.  Considering the importance of the fact to our construction, we would like to reproduce his proof in our settings.\end{rmk}

\begin{prop}\label{cofiltered}{\rm \cite[Chapter II, Proposition 1 and Proposition 2]{Nori} } Fibred products exist in $N(X/k,x)$.
\end{prop}

\begin{proof} We have to show that given any two morphisms $$(\phi_i,h_i): (P_i,G_i,
p_i)\to(Q,G,q)\in N(X/k,x)$$ where $i= 1,2$, the triple
$(P_1\times_QP_2, G_1\times_GG_2, p_1\times_q p_2)$ is again an object in
 $N(X/k,x)$.

The action of $G_1$ on $P_1$ (resp. $G_2$ on $P_2$) induces a morphism of $k$-schemes
$$\lambda: (P_1\times_QP_2)\times_k(G_1\times_GG_2)\to
(P_1\times_QP_2)\times_X(P_1\times_QP_2)$$$$\ \ \ \ \ (x_1,x_2)\times(g_1,g_2)\mapsto (x_1,x_2)\times(x_1g_1,x_2g_2).$$
By a purely abstract nonsense argument, we see that the induced morphism is an isomorphism. Now the problem is to show that the projection $\phi: P_1\times_QP_2\to X$ is FPQC.

Let $Y$ be the quotient of $P_1\times_QP_2$ by $G_1\times_GG_2$,
$$\varphi:P_1\times_QP_2\to Y$$ be the quotient map. Then there is a unique
morphism of schemes $i:Y\to X$ through which the projection $\phi$ factors. Consider the
following commutative diagram:
$$\xymatrix{(P_1\times_QP_2)\times_k(G_1\times_GG_2)\ar[r]^-{\lambda}\ar[r]_-{\cong}\ar[d]_-{\varphi\circ pr_1}
&(P_1\times_QP_2)\times_X(P_1\times_QP_2)\ar[d]^-{\varphi\times
\varphi}\\Y\ar[r]^-{\Delta}&Y\times_XY}.$$ As $i:Y\to X$ is finite, $\Delta$ is of finite presentation \cite[1.4.3.1, pp.231]{EGA IV-1}. Since $\varphi$ is finite faithfully
flat \cite[Expos\'e V, Th\'eor\`eme 4.1, pp.259]{SGA3}, $\varphi\circ pr_1$, $\varphi\times\varphi$, $\lambda$ are all finite and
faithfully flat. So $\Delta$ is also  faithfully flat. But $\Delta$
is already a closed immersion, so it has to be an isomorphism. Hence the
finite morphism $i: Y\to X$ is a monomorphism \cite[17.2.6]{EGA IV-4} in the category of
schemes. Thus it has to be a closed immersion \cite[18.12.6]{EGA IV-4}. Now look at the
following diagram
$$\xymatrix{P_1\times_QP_2\ar@{^{(}->}[r]\ar@{>>}[d]_-{\varphi}&P_1\times_XP_2\ar@{>>}[d]^-{\psi}\\Y\ar[r]^-{i}&X}.$$
Since $P_1\times_QP_2$ is the fibre of the neutral element of $G$
under the following map
$$P_1\times_XP_2\xrightarrow{(\phi_1\times\phi_2)}Q\times_XQ\xrightarrow{\cong}Q\times_kG\xrightarrow{pr_2}G,$$
$P_1\times_QP_2\subseteq P_1\times_XP_2$ must be both open and closed as a sub topological space (but not as a subscheme). The map
$\psi$ is finite flat and of finite presentation, so the underlying topological space of the scheme $Y$, as the image of
$P_1\times_QP_2$ under $\psi$, is both open and
closed in $X$. Since $P_1\times_QP_2$ admits a morphism from a \textit{non-empty} scheme $S$, it must be non-empty as well. Thus $Y\neq \varnothing$. Combining this with the condition that $X$ is connected and reduced we conclude that $i:Y\to X$ is an
isomorphism. Now  $\phi=i\circ\varphi$ is finite locally free and surjective, so in particular FPQC.
\end{proof}

\begin{rmk}\label{rmk}  (i) Proposition \ref{cofiltered}  implies that $N(X/k,x)$ is cofiltered\footnote{This was suggested to us by Jilong Tong. We were using a non-standard notion of cofilteredness in the earlier version. We  thank him for this suggestion.}. Indeed, conditions (i), (ii) of Definition \ref{cofilter} are directly checked. For (iii), suppose we have two maps $a,b: j\to i$ as in \ref{cofilter},  then we could make the following cartesian diagram$$\xymatrix{k\ar[r]^{c}\ar[d]&j\ar[d]^{a\times b}\\i\ar[r]^-{\Delta}&i\times i}$$ where $\Delta$ stands for the diagonal map. The map $c$ in the diagram is precisely what we are looking for.

(ii)  It is rather important that we require $S\neq\emptyset$, otherwise the category is not cofiltered. For example, let's take $X=\Spec(k)$ to be a field, $Q=(\Z/2\Z)_k$ be the trivial torsor under the constant group scheme $(\Z/2\Z)_k$, $P_1=P_2=\Spec(k)$ be the trivial torsor under the trivial $k$-group scheme $\{1\}$,  $\phi_i: P_i\to Q\ (i=1,2)$ be two maps sending $P_i$ to the two different points of $Q$. If the category was cofiltered, then there should be two morphisms of torsors $\psi_i: P\to P_i\ (i=1,2)$ which  equalize $\phi_1$ and $\phi_2$. But $P_1\times_QP_2=\emptyset$, so this can't happen.

(iii) In \ref{cofilter} the reducedness and connectedness assumptions are actually quite important. For example, we could take $X:=\alpha_{p,k}$ where $k$ is a field of characteristic $p$, $Q:= X\times_k\alpha_{p,k}$ the trivial $\alpha_{p,k}$-torsor over $X$, $P_1=P_2=X$ is the trivial torsor over $X$ under the trivial group scheme $\{1\}$. Let $\phi_1: P_1\to Q$ be the diagonal map $X\to X\times_k\alpha_{p,k}=X\times_kX$, and let $\phi_2: P_2\to Q$ be the map $id\times 0:X\to X\times_k\alpha_{p,k}$. Then $P_1\times_QP_2=\Spec(k)$ and the projection $P_1\times_QP_2\to X$ is just the imbedding of the identity point $0:\Spec(k)\hookrightarrow \alpha_{p,k}$ which certainly can't be flat. Now let's equip $X$ with a geometric point $x:\Spec(\bar{k})\rightarrow\Spec(k)\xrightarrow{0}X=\alpha_{p,k} $, then there are unique liftings $p_1\in P_1(\bar{k})$, $p_2\in P_2(\bar{k})$, $q\in Q(\bar{k})$ of $x$. If we had a triple $(T,H,t)\in N(X/k,x)$ and a commutative diagram $$\xymatrix{(T,H,t)\ar[r]\ar[d]&(P_1,\{1\},p_1)\ar[d]^{(\phi_1,0)}\\(P_2,\{1\},p_2)\ar[r]^{(\phi_2,0)}&(Q,\alpha_{p,k},q)}$$ then the structure map $T\to X$ of the $X$-scheme $T$ would factor through $0:\Spec(k)\to X=\alpha_{p,k}$ because $\Spec(k)=P_1\times_QP_2$. Then $T\to X$ can not be faithfully flat, but this contradicts to the assumption that $T$ is a torsor over $X$. Thus $N(X/k,x)$ is not cofiltered. If $X$ is allowed to be non-connected, then take any morphism of $k$-schemes $S\to Y$ with $S$ non-empty, let $Y_1=Y_2=Y_3=Y_4=Y_5=Y_6=Y$, $X:=Y_1\coprod Y_2$, $P_1=P_2=Y_1\coprod Y_2=X$  be the trivial $\{1\}$-torsor, $Q=Y_3\coprod Y_4\coprod Y_5\coprod Y_6$ be the trivial $(\Z/2\Z)_k$-torsor with structure map $Y_3\mapsto Y_1,Y_4\mapsto Y_2, Y_5\mapsto Y_1,Y_6\mapsto Y_2$. Now set $x:  S\to Y_1\subseteq X$ $p_1: S\to Y_1\subseteq P_1$, $p_2: S\to Y_1\subseteq P_2$,  $q: S\to Y_3\subseteq Q$. Let $\phi_1: P_1\to Q$ be the map sending $Y_1\mapsto Y_3$, $Y_2\mapsto Y_4$, $\phi_2: P_2\to Q$ be the map sending $Y_1\mapsto Y_3$, $Y_2\mapsto Y_6$. If we had a triple $(T,H,t)\in N(X/k,x)$ and a commutative diagram $$\xymatrix{(T,H,t)\ar[r]\ar[d]&(P_1,\{1\},p_1)\ar[d]^{(\phi_1,0)}\\(P_2,\{1\},p_2)\ar[r]^{(\phi_2,0)}&(Q,\alpha_{p,k},q)}$$
Then the structure map $T\to X$ of the $X$-scheme would factor through $Y_1\subsetneq X$. Therefore, $T$ is flat but \textit{not} faithfully flat over $X$, a contradiction! So $N(X/k,x)$ can not be cofiltered.

(iv) In \ref{cofiltered}, if $G_1$ and $G_2$ are \'etale then $\phi$ is automatically FPQC (\cite[Expos\'e I, Corollaire 4.8., pp.4]{SGA1}) even when $X$ is non-reduced. However, connectedness is still vital.

\end{rmk}

\begin{defn}\label{projlimit} Let $X$ be a reduced connected  scheme over a  field $k$, $x:S\to X$ be a  morphism of $k$-schemes with $S$ non-empty, $I(X/k,x)\subseteq N(X/k,x)$ be a  cofiltered full subcategory. The forgetful functor $i:=(P_i,G_i,p_i)\longmapsto G_i$ from $I(X/k,x)$ to the category of $k$-group schemes defines a small cofiltered projective system of finite $k$-group schemes. We define the arithematic Nori fundamental group scheme $\pi^I(X/k,x)$ to be $\pi^I(X/k,x):=\varprojlim_{i\in I(X/k,x)}G_i$.
\end{defn}


\section{First Properties of $\pi^I(X/k,x)$}
\subsection{The Universal Covering}
As in \cite[Chapter II, Proposition 2]{Nori} we can define
the universal covering for our fundamental group scheme.

\begin{prop}\label{unicov} Let $X$ be a  connected reduced scheme
over a field $k$, $x:S\to X$ be a  morphism of $k$-schemes with $S$ non-empty,
$I(X/k,x)\subseteq N(X/k,x)$ be a  cofiltered full subcategory. Then there exists a
triple $(\widetilde{X_x},\pi^I(X/k,x),\tilde{x})$, where
 $\widetilde{X_x}$ is a
$\pi^I(X/k,x)$-torsor over $X$,
 $\tilde{x}:S\to\widetilde{X_x}$ is an  $S$-point of $\widetilde{X_x}$ lying above $x$, which satisfies that for any $(P,G,p)\in I$
  there exists a unique morphism $$(\phi,h):
(\widetilde{X_x},\pi^I(X/k,x),\tilde{x})\to (P,G,p), $$where $h:
\pi^I(X/k,x)\to G$ is homomorphism of $k$-group schemes and
$\phi:\widetilde{X_x}\to P$ is a morphism of $X$-schemes which sends
$\tilde{x}$ to $p$ and intertwines the group actions.
\end{prop}
\begin{proof} Consider the following functors $$F_X: I(X/k,x)\to \Aff(X),\ \ \ \ \
(P,G,p)\mapsto P$$$$F_k: I(X/k,x)\to {\rm Grsch}(k),\ \ \ \ \
(P,G,p)\mapsto G$$
where $\Aff(X)$ denotes the category of affine schemes over $X$, and ${\rm Grsch}(k)$ denotes the category of finite group schemes over $k$. We have by \ref{projlimit} that $$\pi^I(X/k,x)=\varprojlim_{i\in I(X/k,x)}F_k(i)$$ Now let
$$\widetilde{X_x}:=\varprojlim_{i\in I(X/k,x)}F_X(i).$$
Then $\widetilde{X_x}$ is an affine scheme over $X$ which admits a point
$\tilde{x}:S\to \widetilde{X_x}$ lying above $x$.

Now we get a triple $(\widetilde{X_x},\pi^I(X/k,x),\tilde{x})$ which
has the property that for any $i:=(P,G,p)\in I(X/k,x)$ there is a morphism
$$(\phi_i,h_i):(\widetilde{X_x},\pi^I(X/k,x),\tilde{x})\to(P,G,p)$$
defined by the projection to the index $i\in I(X/k,x)$. Let $H$ be the image of $h_i$, then we get a factorization of $h_i$ $$\xymatrix{\pi^I(X/k,x)\ar@{->>}[rr]^-{f}&&H\ar@{^{(}->}[rr]^-{g}&&G}$$
and a commutative diagram
 $$\xymatrix{
                & (\widetilde{X_x},\pi^I(X/k,x),\tilde{x}) \ar[dr]^-{(\phi_i,h_i)} \ar@{->>}[dl]_-{(\psi,f)}            \\
 (Q,H,q)  \ar@{^{(}->}[rr]^{(\varphi,g)} & &     (P,G,p)        },$$
where $Q:=\widetilde{X_x}\times_{f}^{\pi^I(X/k,x)}H$ is the contracted producted along $f$, and $\psi,\varphi$ are canonical maps induced by the contracted product. Let $j:=(Q,H,q)$. There is a projection map $$(\phi_j,h_j): (\widetilde{X_x},\pi^I(X/k,x),\tilde{x}) \longrightarrow (Q,H,q)$$ which also makes the above diagram commutative after replacing $(\psi,f)$ by $(\phi_j,h_j)$. Since $\varphi$ and $g$ are closed imbeddings, we must have $(\psi,f)=(\phi_j,h_j)$. Hence the affine ring of $\pi^I(X/k,x)$ is in fact a filtered inductive limit of its sub Hopf-algebras  which are induced by those $j\in I(X/k,x)$ whose $h_j$ are surjective, and the same thing happens for $\widetilde{X_x}$. This implies that if there is
another morphism
$$(\phi,h):(\widetilde{X_x},\pi^I(X/k,x),\tilde{x})\to(P,G,p)$$
 then we
can find an index $i':=(P',G',p')\in I$ such that $\phi, h$ factor through
the projection morphisms
$$\phi_{i'}:\widetilde{X_x}\twoheadrightarrow P'\ \ \ \ \ \  \ \ \ \text{and} \ \ \ \ \ \ \ \ \ \ \ h_{i'}: \pi^I(X/k,x)\twoheadrightarrow G',$$ in other words, we have a commutative diagram $$\xymatrix{
                & (\widetilde{X_x},\pi^I(X/k,x),\tilde{x}) \ar[dr]^-{(\phi,h)} \ar[dl]_-{(\phi_{i'},h_{i'})}            \\
 (P',G',p')  \ar[rr]^{(\varphi,g)} & &     (P,G,p)        }.$$
 But by the very definition of a projective limit, we know that
 $(\varphi,g)\circ(\phi_{i'},h_{i'})=(\phi_i,h_i)$. Thus
 $(\phi_i,h_i)=(\phi,h)$. This completes the proof.
\end{proof}

\begin{cor}\label{imp}Let $\Hom_{\text{\rm grp.sch}}(\pi^N(X/k,x), -)$ be the category whose objects are finite $k$-group schemes equipped with $k$-group scheme homomorphisms from $\pi^N(X/k,x)$, and whose morphisms are $k$-group scheme  homomorphisms which are comptible with the homomorphisms from $\pi^N(X/k,x)$. Then there is an equivalence of categories $$\Hom_{\text{\rm grp.sch}}(\pi^N(X/k,x), -)\xrightarrow{\cong} N(X/k,x).$$  A similar statement holds if one replaces $N(X/k,x)$ by some smaller cofiltered subcategory. \end{cor}

\begin{proof} Given   a $k$-group scheme homomorphism $f:\pi^N(X/k,x)\to G$,  we get a contracted product $$(\widetilde{X_x}\times_{f}^{\pi^N(X/k,x)}G, G,\tilde{x})\in N(X/k,x),$$ and given a morphism in $\Hom_{\text{\rm grp.sch}}(\pi^N(X/k,x), -)$, we get a morphism in $N(X/k,x)$ defined by the universal property of the contracted product. This defines a functor $$\Hom_{\text{\rm grp.sch}}(\pi^N(X/k,x), -)\xrightarrow{\cong} N(X/k,x).$$ The quasi-inverse of this functor is given by \ref{unicov}.
\end{proof}

\begin{defn}\footnote{The terminology \textit{saturated} is taken from \cite{EHV}. We also used it in \cite{Zh}. In \cite{Nori} such objects are called \textit{reduced}.}
Let $X$ be a reduced connected  scheme over a  field $k$,  $x:S\to X$ be a  morphism of $k$-schemes with $S$ non-empty, $I(X/k,x)\subseteq N(X/k,x)$ be a  cofiltered full subcategory. We call a triple $(P,G,p)\in I(X/k,x)$ an
$I$-saturated object if the corresponding projection map $\pi^I(X/k,x)\to G$ is
surjective.
\end{defn}

\begin{lem} \label{saturated}Let $X$ be a reduced connected  scheme over a  field $k$,  $x:S\to X$ be a  morphism of $k$-schemes with $S\neq\emptyset$, $I(X/k,x)\subseteq N(X/k,x)$ be a  cofiltered full subcategory.  Then the full subcategory of $ I(X/k,x)$ consisting of $I$-saturated objects is cofinal in $I(X/k,x)$,  i.e. for any object $(P,G,p)\in I(X/k,x)$ there is a morphism $$(Q,H,q)\to (P,G,p)\in I(X/k,x)$$ where $(Q,H,q)$ is an $I$-saturated object. So when we study projective limits indexed by $I(X/k,x)$ we can restrict ourselves to this smaller category of $I$-saturated objects.
\end{lem}

\begin{proof} Given a triple $(P,G,p)\in I(X/k,x)$ we get a  homomorphism $\pi^I(X/k,x)\to G$. Since  $\pi^I(X/k,x)$ and  $G$ are affine group schemes,  there is a unique decomposition $$\pi^I(X/k,x)\twoheadrightarrow H\subseteq G.$$ By \ref{imp} we have a morphism $(Q,H,q)\subseteq (P,G,p)\in  I(X/k,x)$, where $(Q,H,q)$ corresponds to the surjection $\pi^I(X/k,x)\twoheadrightarrow H$. This finishes the proof.  \end{proof}

\subsection{Relations among $\pi^N,\ \pi^L,\ \pi^E,\ \pi^G$}

\begin{defn} There are various choices of $I(X/k,x)\subseteq N(X/k,x)$. We will list some of them which will be frequently used in the rest of this paper.
\begin{enumerate}[(i)]\item
$\pi^N(X/k,x):=\pi^I(X/k,x)$ when $I(X/k,x)=N(X/k,x)$;\\
\item
 $\pi^E(X/k,x):=\pi^I(X/k,x)$ when $I(X/k,x)=I_{\text{\rm\'et}}(X/k,x)$ is the subcategory consisting of triples $(P,G,p)$ where $G$ is an \textit{\'etale group scheme} over $k$;\\
\item
 $\pi^G(X/k,x):=\pi^I(X/k,x)$ when $I(X/k,x)=I_{\text{\rm co}}(X/k,x)$ is the subcategory consisting of triples $(P,G,p)$ where $G$ is a \textit{constant group scheme} over $k$;\\
\item
 $\pi^L(X/k,x):=\pi^I(X/k,x)$ when $I(X/k,x)=I_{\text{\rm lc}}(X/k,x)$ is the subcategory consisting of triples $(P,G,p)$ where $G$ is a \textit{local} (i.e. \textit{connected}) \textit{group scheme}.
\end{enumerate}
\end{defn}
\begin{rmk}\label{constant2} (i) As we have seen in \ref{rmk} (iv), $\pi^E(X/k,x)$ can be defined without the assumption that $X$ is reduced.

(ii) When $x:S\to X$ is taken to be a geometric point in $ X(\bar{k})$, $\pi^G(X/k,x)$ is a profinite affine group scheme whose group of  $k$-points is just Grothendieck's \'etale fundamental group  $\pi_1^{\rm\text{\'et}}(X,x)$. The only difference between $\pi_1^{\text{\'et}}(X,{x})$ and $\pi^G(X/k,{x})$ is that $\pi_1^{\text{\'et}}(X,{x})$ is a projective limit of finite groups, where the limit is taken in the category of topological groups in which each finite group has the discrete topology while $\pi^G(X/k,{x})$  is a projective limit of finite groups, where the limit is taken in the category of affine   group schemes in which each finite group is regarded as a constant group scheme over $k$. In other words, $\pi^G(X/k,x)$ is none other than a linearization of $\pi_1^{\rm\text{\'et}}(X,x)$.
\end{rmk}

\begin{lem}\label{relation} Let $I_1\subseteq I_2\subseteq N(X/k,x)$ be two cofiltered full subcategories and $(P,G,p)$ be
an object in $I_1$. If for any imbedding $$(Q,H,q)\hookrightarrow
(P,G,p)\in I_2$${\rm(}i.e. $H\subseteq G$ is a subgroup{\rm)}, $(P,G,p)\in I_1$ implies $(Q,H,q)\in I_1$, then we have
a surjection
$$\pi^{I_2}(X/k,x)\twoheadrightarrow \pi^{I_1}(X/k,x).$$
\end{lem}
\begin{proof} Let $(P,G,p)\in I_1$ be an $I_1$-saturated object. Then we can take the image of the composition $$\pi^{I_2}(X/k,x)\to \pi^{I_1}(X/k,x)\twoheadrightarrow G$$ and denote it by $H$. By \ref{imp} we get an inclusion $(Q,H,q)\hookrightarrow (P,G,p)\in I_2$. So by the assumption this inclusion lives in $I_1$. This implies that the surjection  $\pi^{I_1}(X/k,x)\twoheadrightarrow G$ factors through $H\hookrightarrow G$. Thus $H=G$. This concludes the proof.
\end{proof}

\begin{prop}\label{comparison} The following natural $k$-group scheme homomorphisms
\begin{enumerate}[{\rm (i)}]\item $\pi^{N}(X/k,x)\twoheadrightarrow \pi^{E}(X/k,x)\twoheadrightarrow \pi^{G}(X/k,x)$ \item$\pi^{N}(X/k,x)\twoheadrightarrow \pi^{L}(X/k,x)$ \item$\pi^{N}(X/k,x)\twoheadrightarrow \pi^{E}(X/k,x)\times_k\pi^{L}(X/k,x)$\end{enumerate} are all surjections.
\end{prop}

\begin{proof} In the view of \ref{relation}, only the last statement needs to be explained.  For this, we take in  \ref{relation} $I_2:=N(X/k,x)$ and $I_1$ to be the triples $(P,G,p)$ whose group $G$ is isomorphic to a direct product of an \'etale $k$-group scheme and a local $k$-group scheme, i.e.  $G=G^0\times_k G_{\text{\rm\'et}}$. Now suppose $H\subseteq G$ is a subgroup scheme. Then the connected-\'etale sequence for $H$ splits because $H_{\rm red}\subseteq G_{\rm red}=G_{\text{\'et}}\Rightarrow H_{\rm red}=H_{\text{\'et}}$. But since $G_{\text{\'et}}$ acts trivially on $G^0$ and the action of $H_{\text{\'et}}$ on $H^0$ is compatible with that of $G_{\text{\'et}}$ on $G^0$, $H_{\text{\'et}}$ must act trivially on $H^0$, or in other words, $H=H^0\times_k H_{\text{\'et}}$.
\end{proof}
\begin{ex}\label{ex}Here we want to point out that all the above surjections are, in general, not isomorphisms.\\
(i). For $\pi_G^E:\pi^{E}(X/k,x)\twoheadrightarrow \pi^{G}(X/k,x)$. Let's take $X=\Spec(k)=\Spec(\Q)$, $\bar{x}: \Spec(\bar{\Q})\to \Spec(\Q)$ is the natural field extension. Let $\alpha\in \Q$, $n\in\N\setminus\{0\}$, and suppose $x^n-\alpha$ has no root in $\Q$, then $P:=\Spec(\Q[x]/(x^n-\alpha))$ is a non-trivial $\mu_n$-torsor over $\Q$. Choosing any point $p\in P(\bar{\Q})$, we get a triple $(P,\mu_n,p)\in N(X/k,\bar{x})$. Let $\varphi:\pi^{E}(X/k,\bar{x})\to \mu_n$ be the homomorphism
corresponding to $(P,\mu_n, p)$ as in \ref{imp}. If the map $\pi_G^E$ was an isomorphism, then there should be a $k$-group scheme homomorphism $\phi:\pi^{G}(X/k,\bar{x})\to \mu_n$ satisfying $\phi\circ\pi_G^E=\varphi$. But since $\pi^{G}(X/k,\bar{x})$ is a cofiltered projective limit of finite constant group schemes, there must be a factorization $$\xymatrix@R=0.5cm{
                &         H \ar@{.>}[dd]^{\lambda}     \\
  \pi^{G}(X/k,\bar{x}) \ar@{>>}[ur]^{} \ar[dr]_-{\phi}                 \\
                &         \mu_n                 },$$
                where $H$ is a constant group scheme. But when $n$ is a prime number, $\mu_n$ is a $\Q$-scheme of two connected components. Thus the fact that $P$ is a non-trivial torsor would imply that $\varphi$ is surjective, and so is $\phi$. Therefore, the map $\lambda: H\to \mu_n$ should also be sujective, and hence  $\mu_n$ has to be a constant group scheme. But this is  not the case when $n>2$.\\
 (ii). For $\pi_E^N:\pi^{N}(X/k,x)\twoheadrightarrow \pi^{E}(X/k,x)$. If it was an isomorphism then any torsor with local group scheme will be dominated by an \'etale torsor, then the local torsor has to be trivial. Hence any non-trivial local torsor gives a counterexample. Yet we would like to point out that if $X=\Spec(k)$ where $k$ is perfect, $\bar{x}:\Spec(\bar{k})\to \Spec(k)$ is the natural field extension, then $\pi_E^N$ is an isomorphism (see \ref{perfectlemma}). But if $k$ is not perfect and ${\rm char}(k)=p$, one can choose $\alpha\in\bar{k}$ such that $\alpha\notin k$ but $\alpha^p\in k$. Thus the field extension $k\subseteq k(\alpha)$ is a non-trivial $\mu_p$-torsor.\\
(iii). For $\pi_L^N:\pi^{N}(X/k,x)\twoheadrightarrow \pi^{L}(X/k,x)$. As in (ii) any non-trivial \'etale torsor provides a counterexample. And also (iii) is implied by (iv).\\
(iv). For $\pi^{N}(X/k,x)\twoheadrightarrow \pi^{E}(X/k,x)\times_k\pi^{L}(X/k,x).$ There is a perfect counterexample in \cite{EPS}[Remark 4.3].
\end{ex}

\subsection{The Nori-Galois group of a Field}\label{field}

\begin{defn} Let $k$ be a field, $\bar{x}$ be the map  $\Spec(\bar{k})\to \Spec(k)$ coresponding to the natural field extension $k\subseteq \bar{k}$. We call $\pi^N(k/k,\bar{x})$ the Nori-Galois group of $k$.
\end{defn}

\begin{prop}\label{perfectlemma} Let $k$ be a perfect field, $\bar{x}:\Spec(\bar{k})\to \Spec(k)$ be the natural field extension $k\subseteq \bar{k}$. Then the canonical surjection
$$\pi_E^N:\pi^N(k/k,\bar{x})\longrightarrow \pi^E(k/k,\bar{x})$$ is an
isomorphism.
\end{prop}
\begin{proof} Let $(P,G,p)\in N(k/k,\bar{x})$ be an object. Then there is a canonical
isomorphism $P\times_kG\cong P\times_kP.$ Let $P_{\rm red}$ be the
reduced closed subscheme of $P$ and $G_{\rm red}$ be the reduced
closed subscheme of $G$. As $k$ is perfect, $$P_{\rm red}\times_kG_{\rm red}\subseteq P\times_kG\ \ \ \ \ \text{and}\ \ \ \ \  P_{\rm red}\times_kP_{\rm red}\subseteq P\times_kP$$ are the unique reduced closed subschemes  of the underlying spaces. This  induces a diagram
$$\xymatrix{P_{\rm red}\times_kG_{\rm red}\ar[r]\ar[d]& P_{\rm red}\times_kP_{\rm red}\ar[d]\\P\times_kG\ar[r]^{\cong}&
P\times_kP}$$ in which the upper horizontal arrow is an isomorphism. But $G_{\rm red}$ is \'etale, as $k$ is perfect. Therefore, we get a morphsim $$(P_{\rm red},G_{\rm red}, p)\subseteq (P,G,p)\in N(k/k,\bar{x})$$ where $(P_{\rm red},G_{\rm red}, p)\in I_{\text{\rm \'et}}(k/k,x)$. Hence $I_{\text{\rm \'et}}(k/k,x)$ is cofinal inside $N(k/k,x)$. Thus $\pi_E^N$ is an isomorphism.
\end{proof}

\begin{cor} \label{greatcor} Assumptions and notations being as in \ref{perfectlemma}, we have $$\pi^L(k/k,\bar{x})=\{1\}.$$
\end{cor}

\begin{proof} Let $(P,G,p)\in I_{\text{lc}}(k/k,\bar{x})$ be an object. Then, as in the proof of \ref{perfectlemma}, we see that there is an imbedding $$(P_{\rm red},G_{\rm red},p) \subseteq (P,G,p)\in N(k/k,\bar{x}).$$ But since $G$ is connected, $(P_{\rm red},G_{\rm red},p)$ is just the trivial triple. This finishes the proof.
\end{proof}

\begin{prop}\label{locallemma} Let $k$ be a separably closed field, and $\bar{x}:\Spec(\bar{k})\to \Spec(k)$ be the natural field extension. Then we have $$\pi^E(k/k,\bar{x})=\{1\},$$ and  the canonical surjection
$$\pi_L^N:\pi^N(k/k,\bar{x})\longrightarrow \pi^L(k/k,\bar{x})$$ is an
isomorphism.
\end{prop}

\begin{proof} Let $(P,G,p)\in N(k/k,\bar{x})$ be an object, $G_{\text{\'et}}$ be the maximal \'etale quotient of $G$. Then the quotient map $h: G\twoheadrightarrow G_{\text{\'et}}$ induces, by  \ref{imp}, a triple $(P_{\text{\'et}},G_{\text{\'et}},p)\in I_{\text{\'et}}(k/k,\bar{x})$ and a morphism $$(\phi,h): (P,G,p)\twoheadrightarrow (P_{\text{\'et}},G_{\text{\'et}},p)\in N(k/k,\bar{x}).$$ Since $P_{\text{\'et}}$ is an \'etale scheme over a separably closed field, every point of $P_{\text{\'et}}$ is a $k$-rational point. This means that $P_{\text{\'et}}$ is a trivial $G_{\text{\'et}}$-torsor, and hence $\pi^E(k/k,\bar{x})=\{1\}$. Now we can pull back the map $\phi: P\to P_{\text{\'et}}$ along the $k$-rational point $p\in P_{\text{\'et}}(k)$. Then we get a triple $(P^0,G^0,p)\in I_{\text{lc}}(k/k,\bar{x})$ and a morphism $$(P^0,G^0,p)\hookrightarrow (P,G,p)\in N(k/k,\bar{x}).$$ This means that $I_{\text{lc}}(k/k,\bar{x})$ is cofinal inside $N(k/k,\bar{x})$. By the same argument as in \ref{perfectlemma}, we see that $\pi_L^N$ is an isomorphism.
\end{proof}

\begin{prop}\label{projectivelemma} Let $k$ be a field, $X$ be a complete rational variety over $\bar{k}$, $n\in\N^+$, $x:S\to X$ be any morphism with $S$ connected and non-empty. Then we have $$\pi^N(X/k,x)=\pi^E(X/k,x)=\pi^L(X/k,x)=\pi^G(X/k,x)=\{1\}.$$
\end{prop}

\begin{proof} Let $(P,G,p)\in N(X/k,x)$ be an object. Then by \cite[Chapter II, lemma, pp. 92]{Nori} plus K\"unneth formula \cite[Theorem 2.3]{MS}, $P$ is a trivial $G$-torsor, i.e. $P\cong X\times_kG$. Since $S$ is connected, it is mapped to a connected component $Q$ of $X\times_kG$ via $p: S\to P$. As $X\times_kG\cong X\times_{\bar{k}}\bar{G}$, the composition $$Q_{\rm red}\subseteq Q\subseteq X\times_kG\to X$$  must be an isomorphism, thus the map $p$ factors through a section of  the structure map $P\to X$. This means that there is a unique morphism $$(X,\{1\},x)\to (P,G,p)\in N(X/k,x).$$ Therefore $(X,\{1\},x)$ is a cofinal object in $N(X/k,x)$. By \ref{projlimit}, $\pi^N(X/k,x)=\{1\}$.
\end{proof}
\begin{rmk} The connectedness assumption on $S$ in the above proposition is quite important. Otherwise, we could take $x: P\to X$ to be the natural projection, $P:=\P_{\bar{k}}^1\coprod\P_{\bar{k}}^1$ to be the trivial torsor under $\Z/2\Z$, and $p:P\to P$ to be the identity. In this way, there is no morphism $(X,\{1\},x)\to (P,G,p)\in N(X/k,x).$ Thus the homomorphism $\pi^N(X/k,x)\to (\Z/2\Z)_k$ corresponding to $(P,\Z/2\Z,p)$ is not tivial,  but surjective. Therefore, $\pi^N(X/k,x)$ is not trivial.
\end{rmk}

\begin{prop}\label{affinelemma} Let $k$ be a field of characteristic 0, $X:=\A_{\bar{k}}^n$, $n\in\N^+$, $x:S\to X$ be any morphism with $S$ connected and non-empty. Then we have $$\pi^N(X/k,x)=\pi^E(X/k,x)=\pi^L(X/k,x)=\pi^G(X/k,x)=\{1\}.$$
\end{prop}

\begin{proof} The point is that in this case any finite torsor over $X$ is \'etale and any \'etale torsor over $X$ is trivial. Then we do \ref{projectivelemma} again.
\end{proof}

\subsection{The \'Etale Piece of the Arithmetic Fundamental Group Scheme}\label{suprising}

\begin{thm}\label{R} Let $\R$ be the field of real numbers, $\bar{x}:\Spec(\C)\to \Spec(\R)$ be the morphism corresponding to the natural inclusion $\R\subset \C$. Then $$\pi^E(\R/\R,\bar{x})=\varprojlim_{n\in\N^+}\mu_{n,\R}$$ is an infinite   $\R-$group scheme, and the universal covering corresponding to  $\pi^E(\R/\R,\bar{x})$ is a non-Noetherian affine scheme with infinitely many connected components.
\end{thm}
\begin{proof} Let $(P,G,p)\in I_{\text{\'et}}(\R/\R,\bar{x})$. Then $P(\C)$ is a principal homogeous space under $G(\C)$. By Galois descent, there is an action of $\Gal(\C/\R)=\langle \sigma\rangle$ on $P(\C)$ via set-theoretical automorphisms and  an action of $\Gal(\C/\R)=\langle \sigma\rangle$ on $G(\C)$ via group automorphisms such that these two actions are compatible. Let $\sigma(p)=pa$ for some $a\in G(\C)$, $n\in\N^+$ denote the order of $a$. Then for any $b\in G(\C)$, $\sigma(pb)=\sigma(p)\sigma(b)=\ pa\sigma(b).$
But $\sigma^2=id$ is trivial, so $pb=\sigma^2(pb)=\sigma(pa\sigma(b))=\sigma(p)\sigma(a)b=pa\sigma(a)b.$ Thus $a\sigma(a)=e$ is trivial, so  $\sigma(a)=a^{-1}$. Let $Q_n(\C)\subseteq P(\C)$ be the subset $\{pa^i| i\in\N\}$, $H_n(\C)\subseteq G(\C)$ be the subgroup $\langle a\rangle$. These substructures are clearly stable under the $\Gal(\C/\R)$-actions, so they descend to $\R$, i.e. we have a subobject $$\hspace{20pt}(Q_n,H_n,p)\subseteq (P,G,p)\in  I_{\text{\'et}}(\R/\R,\bar{x}),$$ where the set of $\C$-points of $Q_n$ is $Q_n(\C)$ and the group of $\C$-points of $H_n$ is $H_n(\C)$.
Let $$(P_n,\mu_{n,\R},p_n):=(\Spec(\R[x]/(x^{n}+1)),\Spec(\R[x]/(x^{n}-1)),\mathrm{e}^{\frac{(2n-1)\pi i}{n}})\in I_{\text{\'et}}(\R/\R,\bar{x})$$ where 
the action of $\mu_{n,\R}$ on $P_n$ is defined simply by multiplying a $n-$th root of unity on a root of $x^n+1=0$ in $\C$ and $\mathrm{e}^{\frac{(2n-1)\pi i}{n}}$ is  the $n$-th  root $\cos(\frac{(2n-1)\pi i}{n})+i\sin(\frac{(2n-1)\pi i}{n})$.  By sending $a\mapsto\mathrm{e}^{\frac{2\pi i}{n}}$ we get an isomorphism $h:H_n\cong \mu_{n,\R}=\Spec(\R[x]/(x^n-1))$. By sending $p\mapsto\mathrm{e}^{\frac{(2n-1)\pi i}{n}}$   we get an isomorphism of $\R$-schemes $\phi:Q_n\cong\Spec(\R[x]/(x^n+1))$ which is compatible with  $h$ under the actions. This means that the full subcategory of $I_{\text{\'et}}(\R/\R,\bar{x})$ consisting of objects of the form $(P_n,\mu_{n,\R},p_n)$ is cofinal.

On the other hand, the triple $(P_n,\mu_{n,\R},p_n)$ is $ I_{\text{\'et}}$-saturated.  If we have a subobject $$(Q,H,p)\subseteq (P_n,\mu_{n,\R},p_n)\in I_{\text{\'et}}(\R/\R,\bar{x})$$ then $p_n=p\in Q(\C)$ implies $\mathrm{e}^{\frac{\pi i}{n}}\in Q(\C)$ for $Q(\C)$ should always contain  the $\Gal(\C/\R)$-orbit, i.e. the complex conjugation,  of $p=p_n=\mathrm{e}^{\frac{(2n-1)\pi i}{n}}$. Therefore, by the equation$$p_n\mathrm{e}^{\frac{2\pi i}{n}}=\mathrm{e}^{\frac{(2n-1)\pi i}{n}}\cdot \mathrm{e}^{\frac{2\pi i}{n}}=\mathrm{e}^{\frac{\pi i}{n}}$$ we have $\mathrm{e}^{\frac{2\pi i}{n}}\in H(\C)$. Since $H(\C)$ contains the generator of the $n$-th cyclic group $\mu_{n,\R}(\C)$, we have $H(\C)=\mu_{n,\R}(\C)$. Or equivalently, $H=\mu_{n,\R}$ and $Q=P_n$. Thus $(P_n,\mu_{n,\R},p_n)$ is an $ I_{\text{\'et}}$-saturated object.

Now if $m,n\in\N^+$ and $m|n$, then we can define a "raise to $\frac{n}{m}$-power" map $$(P_{n},\mu_{n,\R},p_{n})\to(P_{m},\mu_{m,\R},p_m)$$   by sending $x\mapsto x^{\frac{n}{m}}$ in the affine coordinate ring. This defines a projective system in $I_{\text{\'et}}(\R/\R,\bar{x})$. By taking projective limit in the category of affine schemes (resp. group schemes) over $\R$, we get a triple $$(\varprojlim_{n\in\N^+}P_n, \varprojlim_{n\in\N^+}\mu_{n,\R}, \tilde{p}).$$

Let $(\widetilde{X}_{\bar{x}},\pi^E(\R/\R,\bar{x}),\tilde{x})$ be the universal triple defined in \ref{unicov}. Then by the universality, we get a morphism $$(\widetilde{X}_{\bar{x}},\pi^E(\R/\R,\bar{x}),\tilde{x})\longrightarrow (\varprojlim_{n\in\N^+}P_n, \varprojlim_{n\in\N^+}\mu_{n,\R}, \tilde{p})$$ which is indeed an isomorphism because of the fact that $\{(P_n,\mu_{n,\R},p_n)|n\in\N^+\}$ is  cofinal and saturated in $ I_{\text{\'et}}(\R/\R,\bar{x})$. This proves that $\pi^E(\R/\R,\bar{x})$ is infinite and also that $\widetilde{X}_{\bar{x}}$ has infinitely many connected components. Since $\widetilde{X}_{\bar{x}}$ is affine, it must be quasi-compact. But then the connected components of $\widetilde{X}_{\bar{x}}$ can not be open, otherwise there should be finitely many of them. Therefore $\widetilde{X}_{\bar{x}}$ is not Noetherian.
\end{proof}

\begin{prop}\label{fields} Let $k$ be  field whose Galois group $\Gal(\bar{k}/k)$ admits $\Z/l\Z$ as a quotient for some prime number $l> 3$. Let $X=\Spec(k)$, $\bar{x}:\Spec(\bar{k})\to X$ be a geometric point. Then $\pi^E(k/k,\bar{x})$ is a non-commutative $k$-group scheme.
\end{prop}
\begin{proof}  Let $k\subseteq K\subseteq \bar{k}$  a finite Galois subextension so that $\Gal(K/k)=\langle \sigma\rangle\cong\Z/l\Z $. Let $G_K:=(\Z/l\Z\times \Z/l\Z)\rtimes\langle b\rangle$, where $\langle b\rangle\cong \Z/l\Z$ acts on  $\Z/l\Z\times \Z/l\Z$ by $$b\ \ \ \longmapsto \ \ \ \left(\begin{array}{cc}
1 & 1 \\
0 & 1
\end{array}\right).$$ 
We  define an action of $\Gal(K/k)$ on $G_K$ by letting $\sigma (z)=z$ for all $z\in \Z/l\Z\times \Z/l\Z$ and $\sigma(b)=\left(\begin{array}{c}
0  \\
1 
\end{array}\right)b$. This action corresponds, by Galois descent, to a $k$-group scheme $G$ which is a $k$-form of the   $K$-group scheme $G_K$.

The   constant $K$-group scheme $G_K$ can be written as $$G_K:=\coprod_{i\in G_K}Y_i$$ where $Y_i=\Spec(K)$. $G_K$ acts on itself by right translations, i.e.  for any $j\in G_K$, $j$ acts on $Y_i$ by the identity map $\Spec(K)=Y_i\to Y_{ij}=\Spec(K)$. Now we define a twisted action of $\Gal(K/k)$ on the $K$-scheme $G_K$. We define the action of $\sigma$ on $Y_i$ to be  the morphism $\tau$ in the following commutative diagram $$\xymatrix{Y_i\ar[r]^-{\tau}\ar@{=}[d]&Y_{b\sigma(i)}\ar@{=}[d]\\\Spec(K)\ar[r]^{t_{\sigma^{}}}&\Spec(K)}$$where $t_{\sigma}$ is the map obtained by applying the functor $\Spec(-)$ to the field automorphism $\sigma: K\to K$. We have the following compatiblity between the action of $\Gal(K/k)$ on the $K$-group scheme $G_K$ and that on the $K$-scheme $G_K$, i.e. the diagram $$\xymatrix{Y_i\ar[r]^j\ar[d]_-{\tau}&Y_{ij}\ar[d]^-{\tau}\\Y_{b\sigma(i)}\ar[r]^-{\sigma(j)}&Y_{b\sigma(ij)}}$$ is commutative for any $j\in G_K$. By Galois descent, the $K$-scheme $G_K$ descends to a $k$-scheme $P$ and there is an action of $G$ on $P$ which makes $P$ a $G$-torsor over $k$. Picking  $p\in Y_{e}(\bar{k})$ to be the inclusion $K\subseteq \bar{k}$, we get an object $(P,G,p)\in I_{\text{\'et}}(k/k,\bar{x})$. This object induces a $k$-homomorphism $$\lambda:\pi^E(k/k,\bar{x})\longrightarrow G.$$ Let $N\subseteq G$ be the image. Then we get a subobject $(Q,N,p)\subseteq (P,G,p)$. As $p\in Q$, $Y_e\subseteq Q_K\Rightarrow Y_b=\sigma(Y_e)\subseteq Q_K\Rightarrow b\in N_K\subseteq G_K$. But $N_K\subseteq G_K$ is stable under the Galois action, so  $\sigma(b)=\left(\begin{array}{c}
0  \\
1 
\end{array}\right)b\in N_K\Longrightarrow\left(\begin{array}{c}
0  \\
1 
\end{array}\right)\in N_K\Longrightarrow\left(\begin{array}{c}
1  \\
1 
\end{array}\right)=b\left(\begin{array}{c}
0  \\
1 
\end{array}\right)b^{l-1}\in N_K$. Thus  $N_K=G_K$ for $\{\left(\begin{array}{c}
1  \\
1 
\end{array}\right),\left(\begin{array}{c}
0  \\
1 
\end{array}\right),b\}$ generates $G_K$. Therefore $\lambda$ is surjective. Then $\pi^E(k/k,\bar{x})$ must be non-commutative for $G$ is.
\end{proof}

\begin{rmk} The point of the assumption $l>3$ is that one needs the action of $\sigma^l$ on $G_K$ to be trivial, i.e. one needs that 
$b\sigma(b)\sigma^2(b)\cdots\sigma^{l-1}(b)$ to be trivial in $G_K$. For this one needs $$1^2+2^2+3^2+\cdots+(l-1)^2=\frac{1}{6}(l-1)l(2l-1)$$ to be divisible by $l$. This is OK only when the prime number $l>3$. 
\end{rmk}

\begin{ex}\label{AB} The notion of $\pi_1^{\text{\'et}}(X,x)$ is \textit{absolute}, i.e. it has no reference to the base field, so $X$ could even be a scheme of mixed characteristic. The notion of $\pi^N(X,x)$ in \cite{Nori} depends only on the base field where $X$ is defined. However, the fundamental group we are considering here depends  also on the  field where the group structure is defined. 

 By a theorem of Serre-Lang, it is known that for $X/\bar{k}$ an abelian variety $\pi_1^{\text{\'et}}(X,0)=\varprojlim_{n\in\N^+}X[n](\bar{k})$, or more generally, Nori proved in \cite{Nori2} that $\pi^{N}(X,0)=\varprojlim_{n\in\N^+}X[n]$. Since our fundamental group is a generalization of \cite{Nori}, we still have  $\pi^{N}(X/\bar{k},0)=\pi^{N}(X,0)=\varprojlim_{n\in\N^+}X[n]$.  However, if we see $X$ as a scheme over $k$ via $X\to \Spec(\bar{k})\to \Spec(k)$ then we really get something different. In this example we take  an abelian variety $X$ over $\C$ and view it as a scheme over $\R$,  then show that $\pi^N(X/{\R},0)$ is non-commutative. 

Let $A$ be an abelian variety over $k:=\R$, $K:=\C$, $\bar{x}\in A_K(K)$. Take any Galois covering $Y\to A_{K}$ with Galois group $\Z/2\Z$. Let $G_K:=\langle a\rangle\rtimes \langle b\rangle$, where $\langle a\rangle\cong\Z/n\Z$ for $n\geq 3\in\N^+$  and   $\langle b\rangle \cong \Z/2\Z$ acts on $\langle a\rangle$ by $b(z)=z^{-1}$ for all $z\in \langle a\rangle$. We  define an action of $\Gal(K/k)$ on $G_K$ by letting $\sigma (z)=z^{-1}$ for all $z\in \langle a\rangle$ and $\sigma(b)=ab$. Then there is a  $k$-form $G$  of the   $K$-group scheme $G_K$ which corresponds to this action.

Let $H_K\subset G_K$ denote the subgroup  $\langle a\rangle$, and let $$P_K:=\coprod_{i\in H_K}Y_i$$ where $Y_i=Y$. Now we define an action of $G_K$ on $P_K$. Take any $g\in G_K$, we can write it uniquely as $g=b^rj$, where $r\in\{0,1\}$ and $j\in H_K$, then the action of $g$ on $Y_i$ is defined to be the morphism $\tau$ in the following commutative diagram $$\xymatrix{Y_i\ar[r]^-{\tau}\ar@{=}[d]&Y_{b^{2-r}(i)j}\ar@{=}[d]\\Y\ar[r]^{b^{r}}& Y}$$
where $b^r$ is the map defined by the non-trivial $A_K$-automorphism of $Y$ if $r=1$, the identity if $r=0$. In this way, $P_K$ becomes a $G_K$-torsor over $A_K$. Now viewing $G_K$ as a constant $K$-group scheme we get a morphism $$\rho:P_K\times_{\Spec(K)}G_K\longrightarrow{}P_K$$ defined by the above action. Composing $\rho$ with the following isomorphism $$P_K\times_{\Spec(K)}(\Spec(K)\times_{\Spec(k)}G)\xrightarrow{\cong}P_K\times_{\Spec(K)}G_K$$ we get an action of $G$ on $P_K$ which makes $P_K$ a $G$-torsor over $A_K$. Picking any $k$-morphism $p:\Spec(K)\to Y_{e}=Y$ (where $e\in H_K$ is the identity element) over $\bar{x}$, we get an object $(P_K,G,p)\in I_{\text{\'et}}(A_K/k,\bar{x})$. This object induces a $k$-homomorphism $$\lambda:\hspace{20pt}\pi^E(A_K/k,\bar{x})\longrightarrow G.$$ Let $N\subseteq G$ be the image. Then we get a subobject $(Q,N,p)\subseteq (P_K,G,p)$. As $p\in Q$, $Y_e\subseteq Q$. Thus $b\in N_K\subseteq G_K$ (because $Y_e\subseteq P_K$ is a torsor under $\langle b\rangle\subseteq G_K$). But $N_K\subseteq G_K$ is stable under the Galois action, so  $\sigma(b)=ab\in N_K\Longrightarrow a\in N_K\Longrightarrow N_K=G_K$. Therefore $\lambda$ is surjective. Then $\pi^N(A_K/k,\bar{x})=\pi^E(A_K/k,\bar{x})$ must be non-commutative for $G$ is.
\end{ex}

\subsection{Comparison of the Geometric Fundamental Groups}\label{geometricfundamentalgroup}

Let $X$ be a separable  geometrically connected scheme over a field $k$,  and $\bar{x}:\Spec(\bar{k})\hookrightarrow {X}$ be a geometric point. Associated to the arithmetic fundamental group scheme $\pi^N(X/k,\bar{x})$, there are two geometric fundamental group schemes $\pi^N(\bar{X}/{k},\bar{x})$ and $\pi^N(\bar{X}/\bar{k},\bar{x})$, the later being the \textit{classical} Nori's fundamental group. We would like to understand the relation between these two. 

\begin{prop}\label{comparison0} Notations and assumptions being as above, if $X$ is moreover quasi-compact and $k$ is perfect,  then
we have an imbedding
$$\chi_{X/k}^N: \hspace{20pt}\pi^N({\bar{X}}/\bar{k},\bar{x})\hookrightarrow\pi^N({\bar{X}}/k,\bar{x})\times_k\bar{k}$$
of $\bar{k}-$group schemes. A similar statement holds if one replaces $N$ by $E,G,L$.
\end{prop}

\begin{proof} Given  $(P,G,p)\in N(\bar{X}/k,\bar{x})$,  $(P,G\times_k\bar{k},p)$ is naturally an object in $N(\bar{X}/\bar{k},\bar{x})$.  In this way we get a functor $F$ which makes the following diagram 2-commutative $$\xymatrix{ N(\bar{X}/k,\bar{x})\ar[rr]^-F\ar[dr]_{\varphi}&& N(\bar{X}/\bar{k},\bar{x})\ar[ld]^{\psi}\\&\text{Grsch}(\bar{k})&}$$  where $\varphi$ is the functor sending $(P,G,p)\in  N(\bar{X}/k,\bar{x})$ to $G\times_k\bar{k}$, and $\psi$ is the forgetful functor sending $(Q,H,q)\in  N(\bar{X}/\bar{k},\bar{x})$ to  $H$. Since base change is compatible with taking projective limit, we have $$\varprojlim_{i\in  N(\bar{X}/k,\bar{x})}\varphi(i)=\pi^N(\bar{X}/k,\bar{x})\times_k\bar{k}.$$ Therefore, we get the homomorphism $$\chi_{X/k}^N: \hspace{20pt}\pi^N({\bar{X}}/\bar{k},\bar{x})\longrightarrow\pi^N(\bar{X}/k,\bar{x})\times_k\bar{k}$$ by passing to the limit. The injectivity of $\chi_{X/k}^N$ is proved in \ref{exactness2}.
\end{proof}

\begin{prop}\label{comparison2} If $X$ is a  connected scheme over any field $k$ with a geometric point $\bar{x}\in X(\bar{k})$, and if $X$ is also a $\bar{k}$-scheme (e.g. $X=Y\times_k\bar{k}$ for some $k$-scheme $Y$), then
the  immbedding
$$\chi_{X/\bar{k}/k}^E: \hspace{5pt}\pi^G({X}/k,\bar{x})\times_k\bar{k}=\pi^G({X}/\bar{k},\bar{x})=\pi^E({X}/\bar{k},\bar{x})\hookrightarrow\pi^E({X}/k,\bar{x})\times_k\bar{k}$$
of $\bar{k}-$group schemes is a section of the quotient map$$\pi_G^E\times_k\bar{k}:\hspace{5pt}\pi^{E}({X}/k,\bar{x})\times_k\bar{k}\twoheadrightarrow \pi^{G}({X}/k,\bar{x})\times_k\bar{k}.$$ \end{prop}

\begin{proof}
Let's first redo the construction in \ref{comparison0}. Given  $(P,G,p)\in I_{\text{\'et}}(X/k,\bar{x})$, let $G'$ be the abstract group associated to $G\times_k\bar{k}$. Viewing $G'$ as a constant group scheme over $k$, we get an object $(P,G',p)\in I_{\text{co}}(X/k,\bar{x})$.  In this way we get a functor which makes the following diagram 2-commutative $$\xymatrix{ I_{\text{\'et}}(X/k,\bar{x})\ar[rr]^-F\ar[dr]_{\varphi}&& I_{\text{co}}(X/k,\bar{x})\ar[ld]^{\psi}\\&\text{Grsch}(\bar{k})&}$$  where $\varphi$ is the functor sending $(P,G,p)\in  I_{\text{\'et}}(X/k,\bar{x})$ to $G\times_k\bar{k}$, and $\psi$ is the functor sending $(P,G,p)\in  I_{\text{co}}(X/k,\bar{x})$ to the abstract group $G$ regarded as a group scheme over $\bar{k}$. Since base change is compatible with projective limit, we have $$\varprojlim_{i\in  I_{\text{\'et}}(X/k,\bar{x})}\varphi(i)=\pi^E(X/k,\bar{x})\times_k\bar{k}\hspace{20pt}\text{and}\hspace{20pt}\varprojlim_{i\in  I_{\text{co}}(X/k,\bar{x})}\psi(i)=\pi^G(X/k,\bar{x})\times_k\bar{k}$$ This defines the homomorphism $\chi_{X/\bar{k}/k}^E$ which is then easily seen as a section of $\pi_G^E$, because the (right) composition of $F$ with the inclusion $$i: I_{\text{co}}(X/k,\bar{x})\longrightarrow I_{\text{\'et}}(X/k,\bar{x})$$ is isomorphic to the identity functor on $I_{\text{co}}(X/k,\bar{x})$.
\end{proof}


\begin{rmk} We have seen from \ref{AB} that both $\chi_{X/k}^N$ and $\chi_{X/\bar{k}/k}^E$ are not, in general, isomorphisms.
\end{rmk}

\subsection{The Geometric Base Point}

\begin{prop}\label{basepoint} Let $X$ be any connected reduced scheme over $k$, $\bar{x}_1: \Spec(\bar{l}_1)\to X$ and $\bar{x}_2: \Spec(\bar{l}_2)\to X$ be two  geometric points of $X$. Then there are (non-canonical) isomorphisms between the following $k$-group schemes: \[\pi^{E}({X}/k,\bar{x}_1)\cong \pi^{E}({X}/k,\bar{x}_2)\tag{i}\] \[\pi^{L}({X}/k,\bar{x}_1)\cong \pi^{L}({X}/k,\bar{x}_2)\tag{ii}\]\[\pi^{N}({X}/k,\bar{x}_1)\cong \pi^{N}({X}/k,\bar{x}_2).\tag{iii}\]
\end{prop}
\begin{proof} 

Let's first examine (i). Let $I_{\text{\'et}}(X/k)$ be the category of pairs $(P,G)$, where $P$ is a torsor over $X$ under a finite \'etale $k$-group scheme $G$ and let $\text{Ecov}(X)$ be the category of finite \'etale coverings of $X$. Consider the following functors \[\xymatrix{I_{\text{\'et}}(X/k)\ar[r]^{F}& \text{Ecov}(X)\ar@/^/[r]^-{F_{\bar{x}_1}}\ar@/_/[r]_{F_{\bar{x}_2}}& \text{((Sets))}}\] where $F$ is the forgetful functor (forgetting the group) and $F_{\bar{x}_1}, F_{\bar{x}_2}$ are  the fibre functors induced by ${\bar{x}_1},\bar{x}_2$. The category $I_{\text{\'et}}(X/k,\bar{x}_1)$ is just the opposite category of representable  presheaves over the presheaf $F_{\bar{x}_1}\circ F$ on $I_{\text{\'et}}(X/k)^{\rm o}$, i.e. its objects are  pairs $(A,a)$ where $A\in I_{\text{\'et}}(X/k)$ and $a: A\to F_{\bar{x}_1}\circ F$ is a morphism of presheaves on $I_{\text{\'et}}(X/k)^{\rm o}$. But from \cite[Expos\'e V, Corollaire 5.7, pp. 107]{SGA1}, there is an isomorphism of functors $F_{\bar{x}_1}\cong F_{\bar{x}_2}$, hence an isomorphism $F_{\bar{x}_1}\circ F\cong F_{\bar{x}_2}\circ F$. Therefore we get an equivalence $I_{\text{\'et}}(X/k,\bar{x}_1)\cong I_{\text{\'et}}(X/k,\bar{x}_2)$ which is compatible with the forgetful functors to $I_{\text{\'et}}(X/k)$. This gives the isomorphism (i).

The isomorphism (ii) is clear. The reason is that surjective purely inseparable morphisms are homeomorphisms on the underlying topological spaces.

Finally we consider the sequence  \[\xymatrix{N(X/k)\ar[r]^{q}& I_{\text{\'et}}(X/k)\ar[r]^{F}& \text{Ecov}(X)\ar@/^/[r]^-{F_{\bar{x}_1}}\ar@/_/[r]_{F_{\bar{x}_2}}& \text{((Sets))}}\]
where $N(X/k)$ the category of pairs $(P,G)$ in which $G$ is a finite $k$-group scheme, $P$ is a torsor over $X$ under $G$, and $q$ is the the functor sending any pair $(P,G)$
to its \'etale quotient $(P_{\text{\'et}},G_{\text{\'et}})$. Now replacing $ I_{\text{\'et}}(X/k)$ by $N(X/k)$ and $F$ by $F\circ q$ we can do the same argument as that in the proof of (i) to get the isomorphism (iii). 
\end{proof} 

\begin{rmk}\label{basepoint2} From the proof of (ii) we see that actually in the definition of $\pi^L$ the base point is not necessary, as for any two different base points $\bar{x_1},\bar{x_2}$ of $X$, the isomorphism $\pi^{L}({X}/k,\bar{x}_1)\xrightarrow{\cong} \pi^{L}({X}/k,\bar{x}_2)$ is \textit{canonical}.
\end{rmk}

\subsection{Base Change}
\begin{prop} \label{basechange} Let $X$ be a scheme geometrically connected proper separable
over a field $k$, $k\subseteq l\subseteq l'$ be a sequence of field
extensions, where $l$ and $l'$ are algebraically closed fields. Let
$\bar{x}:\Spec(l')\to X$ be a geometric point. Then the following
natural map
$$\pi_{l}^{l'}:\pi^E(X\times_kl'/k,\bar{x})\longrightarrow\pi^E(X\times_kl/k,\bar{x})$$ is an isomorphism of $k$-group schemes.
\end{prop}
\begin{proof}Let $Y'\to X\times_kl'$ be a $G'$-torsor with a fixed point $\Spec(l')\to Y'$ lying over $\bar{x}$.
By \cite[Expos\'e X, Corollaire 1.7]{SGA1},
$$\pi_1^{\text{\'et}}(X\times_kl',\bar{x})\cong\pi_1^{\text{\'et}}(X\times_kl,\bar{x}).$$
Thus by  \cite[Expos\'e V,Th\'eor\`eme 4.1]{SGA1}, the base change
functor $-\times_ll'$ induces an equivalence of categories between
the categories of finite \'etale coverings ${\rm ECov}(X\times_kl)$ and  ${\rm ECov}(X\times_kl')$. Thus
there is a finite \'etale covering $Y\to X\times_kl$ such that
$Y\times_ll'=Y'.$ Now from the full faithfulness of  $-\times_ll'$ and the
fact that $G\times_kl$ and $G\times_kl'$ are constant group schemes,
the action
$$(Y\times_ll')\times_{l'}(G\times_kl')=Y\times_ll'\times_kG=Y'\times_kG\to Y'=Y\times_ll'$$ descends to
an action $Y\times_kG\to Y$ and makes $Y$ a $G$-torsor. This means that the pull back functor $$F_{l'}^{l}:
N(X\times_kl/k,\bar{x})\to N(X\times_kl'/k,\bar{x})$$ is essentially surjective. But by the fully faithfulness of  $-\times_ll'$ the pull back functor $F_{l'}^{l}$ is also fully faithful.  Hence $F_{l'}^{l}$ is an equivalence, and therefore the canonical morphism $\pi_{l}^{l'}$ is an isomorphism.
\end{proof}
\begin{rmk}\label{basechangeinf} Unfortunately the similar statement for $\pi^L$ is false. This is due to an example by V. Mehta and S. Subramanian. Let $X$ be an integral projective curve over $k=\bar{k}$ of characteristic $p>0$ with at least one cuspidal singularity. Let $x\in X(k)$ be a rational point, and $k\subsetneq l$ be an extension of algebraically closed fields. We have the following commutative diagram $$\xymatrix{\pi^L(X\times_kl/l,x)\ar[rr]\ar[dr]&&\pi^L(X\times_kl/k,x)\times_kl\ar[ld]^-{\pi_k^{l}\times \id}\\&\pi^L(X/k,x)\times_kl&}$$
with canonical morphisms. In \cite[\S 3]{MS}, Mehta and Subramanian constructed a homomorphism $\phi:\pi^L(X\times_kl/l,x)\to \mu_{p,l}$ which does not come from a homomorphism $\pi^L(X/k,x)\to \mu_{p,k}$ by base change. If $\pi_k^l$ was an isomorphism, then $\phi$ does not come from a homomorphism $\pi^L(X\times_kl/k,x)\to \mu_{p,k}$. But this is a contradiction, since any $\mu_{p,l}$-torsor over $X\times_kl$ comes from a $\mu_{p,k}$-torsor over $X\times_kl$.
\end{rmk}

\subsection{The \'Etale Universal Covering}

In this subsection we want to emphasize a big difference between
$\pi^E(X/k,x)$ and $\pi^G(X/k,x)$ (or $\pi_1^{\text{\'et}}(X,x)$) via comparing their universal
coverings. The following statement is well known in the literature. 
\begin{stat} Let $X$ be a connected Noetherian scheme, $x\in X(\Spec(\bar{k}))$ be any geometric point. Then the universal covering $\widetilde{X}_x$ corresponding to $\pi_1^{\text{\'et}}(X,x)$ is connected.
\end{stat}
 The major reason behind this
phenomenon is the following: \begin{fact}\label{conn} If $X$ is a
locally Noetherian connected scheme, and $x\in X(\Spec(\bar{k}))$ is any geometric point, then for any triple $(P,G,p)\in I_{\text{co}}(X,x)$  the
corresponding map $\pi_1^{\text{\'et}}(X,x)\to G$ is surjective if and only if $P$ is
connected.
\end{fact}

But for universal coverings under
$\pi^E$, they are usually highly non-connected. We have seen some examples in \ref{suprising} which are caused by complicated structures of the \'etale group schemes.  Here is another example which is caused by the choice of the point on the torsor.
\begin{ex}\label{euni} Let $X=\Spec(\Q)$, $\bar{x}:\Q\subseteq \bar{
\Q}$. Consider a prime number $p>2$. Then $\mu_p\cong \Spec(\Q)\coprod\Spec(K)$ as a scheme, where
 $K$ is the $p$-th cyclotomic field. Let $(\mu_p,\mu_p,q)$ be the trivial
$\mu_p$-torsor equipped with the point
$q:\Spec(\bar{\Q})\to \Spec(K)$. Obviously $\mu_p$ is not connected, but
the unique map
$$(\phi,h): \ (\widetilde{X_x},\pi^E(X/k,\bar{x}),\tilde{x})\to (\mu_p,\mu_p,q)$$ can
not be trivial on  $h$, for otherwise
$(\phi,h)$ would factor through the trivial triple $(X,\{1\},\bar{x})$, and then $q$
has to be the trivial point $\Spec(\bar{\Q})\to \Spec(\Q)$. But if $h$ is
non-trivial then it has to be surjective. Therefore $(\mu_p,\mu_p,q)$ is saturated but not connected. Since
$h$ is surjective, $\phi$ must be faithfully flat. But $\mu_p$ is not connected so
$\widetilde{X_x}$ can not be either.
\end{ex}

\begin{proof}[Proof of the fact] "$\Rightarrow$" If $P$ was not connected then we can take the connected component $Q\subsetneq P$ containing $p$. Let $H\subsetneq G$ be the stabilizer of $Q$, then $(Q,H,p)\subsetneq (P,G,p)$. Therefore we have a factorization $\pi_1^{\text{\'et}}(X,x)\to H\subsetneq G$ which contradicts to the assumption that $\pi_1^{\text{\'et}}(X,x)\to G$ is surjective. "$\Leftarrow$" Suppose $\pi_1^{\text{\'et}}(X,x)\to G $ factorizes as $\pi_1^{\text{\'et}}(X,x)\to H\subseteq G$. Then we would have an imbedding $(Q,H,p)\subseteq (P,G,p)$. But $Q\subseteq P$ is finite \'etale, so it's both open and closed. Therefore $Q= P$ for $Q\neq \emptyset$. Hence $H=G$.
\end{proof}

\begin{proof}[Proof of the statement] Let $I\subseteq I_{\text{co}}(X,x)$ be the full subcategory consisting of saturated objects. By \ref{saturated}, the category $I$ is cofiltered. Then $\widetilde{X}_x=\varprojlim_{i\in I}P_i$, where $i=(P_i,G_i,p_i)\in I$. Because of the above \textit{Fact}, these $P_i$ are connected. The scheme $\widetilde{X}_x$ is connected if and only if $H^0(\widetilde{X}_x,O_{\widetilde{X}_x})$ has no non-trivial idempotents. 
Since $X$ is quasi-compact and $\varinjlim$ is an exact functor we know that $$H^0(\widetilde{X}_x,O_{\widetilde{X}_x})=\varinjlim_{i\in I}H^0(P_i,O_{P_i}).$$ As each $P_i$ is connected, there is no non-trivial idempotent in $H^0(P_i,O_{P_i})\subseteq H^0(\widetilde{X}_x,O_{\widetilde{X}_x})$, hence there is no non-trivial idempotent in $H^0(\widetilde{X}_x,O_{\widetilde{X}_x})$.
\end{proof}

\section{The First Fundamental  Sequence}

\subsection{The General Case}

\begin{prop}\label{surjectivitylemme}\label{surjectivity2}Let $X$ be a geometrically connected separable scheme  over a field $k$, and $\bar{x}:\Spec(\bar{k})\hookrightarrow X$ be a geometric point. Then the
 natural $k$-group scheme homomorphism $$\pi^{I}(X/k,\bar{x})\to \pi^{I}(k/k,\bar{x})$$ is  surjective for $I=E,G,N,L$.\end{prop}
\begin{proof}Suppose that we have an object $(l,G,t)\in I(k/k)$
and that we have a morphism $$(\lambda,i): (Q,H,s)\to (l\times_kX,G,t)\in I(X/k),$$where the group homomorphism $i:H\to G$ is a closed imbedding. Then we have a section in the category of $X$-schemes $$X=Q/H\hookrightarrow( l\times_kX)/H=(l/H)\times_kX.$$ As $l/H$ is  finite over $k$, its connected components are single points. Let  $x\in l/H$ be the image  of  $t\in l$ under the projection $l\to l/H$. Since $X$ is  connected, reduced and $\lambda$ sends $s\mapsto t$, the map $$X\hookrightarrow( l/H)\times_kX\xrightarrow{pr_1} l/H$$  factors through  $x:\Spec(\kappa(x))\hookrightarrow l/H$ where $\kappa(x)$ is the residue field of $x$.  Hence $X$ is a scheme over $\kappa(x)$. But $X$ is geometrically connected and geometrically reduced over $k$, so the extension $k\subseteq \kappa(x)$ has to be trivial, i.e. $k=\kappa(x)$. In other words,  $x$ is a $k$-rational point of $l/H$. Now pull back the projection map $l\to l/H$ along $x:\Spec(k)\to l/H$,  we get  a map $(q,H,t)\to(l,G,t)\in I(k/k)$ in which the group homomorphism is the imbedding $i:H\hookrightarrow G$.  In particular if the map $\pi^{I}(k/k,\bar{x})\to G$ corresponding to $(l,G,t)$ is surjective, then the composition $$\pi^{I}(X/k,\bar{x})\to\pi^{I}(k/k,\bar{x})\twoheadrightarrow G$$ has to be surjective too. This means precisely that $\pi^{I}(X/k,\bar{x})\to \pi^{I}(k/k,\bar{x})$ is  surjective.
\end{proof}

\begin{prop} \label{exactness} Let $X$ be a geometrically connected separable scheme  over a field $k$, and $\bar{x}:\Spec(\bar{k})\hookrightarrow X$ be a geometric point. Then the
 natural sequence of $k$-group schemes \begin{equation}\label{seq1}1\to\pi^{I}(\bar{X}/k,\bar{x})\to \pi^{I}(X/k,\bar{x})\to  \pi^{I}(k/k,\bar{x})\to 1\tag{1}\end{equation} is a complex, and it  is  exact in the middle  if and only if the following two conditions are satisfied. \begin{enumerate}[(i)]\item For any $I$-saturated object $(P,G,p)\in I(X/k,\bar{x})$, the image of the composition of the natural homomorphisms $$\pi^{I}(\bar{X}/k,\bar{x})\to \pi^{I}(X/k,\bar{x})\twoheadrightarrow G$$ is a normal subgroup of $G$.\item Whenever there is an object $(P,G,p)\in I(X/k,\bar{x})$ whose pull-back along $\bar{X}\to X$ is trivial then there is an object $(Q,H,q)\in I(k/k,\bar{x})$ whose pull-back along $X\to \Spec(k)$ is isomorphic  to $(P,G,p)$.
\end{enumerate}
\end{prop}
\begin{proof} For the first statement it is enough to see that the pull-back functor $$\mathfrak{C}(k/k,\bar{x},I)\to\mathfrak{C}(\bar{X}/k,\bar{x},I)$$ sends any object in $\mathfrak{C}(k/k,\bar{x},I)$ to a trivial object in $\mathfrak{C}(\bar{X}/k,\bar{x},I)$. But this is indeed the case, for the pull-back functor $I(k/k,\bar{x})\to I(\bar{X}/k,\bar{x})$ is trivial.

Now for the second statement. "$\Rightarrow$" (i) is clear, because a normal subgroup is still normal in any quotient. (ii) follows directly from \ref{imp}. Indeed, given $(P,G,p)\in I(X/k,\bar{x})$, there is a unique morphism $\phi:\pi^I(X/k,\bar{x})\to G$ corresponding to $(P,G,p)$. The pull-back of $(P,G,p)$ is trivial means that the composition $$\pi^I(\bar{X}/k,\bar{x})\to\pi^I(X/k,\bar{x})\xrightarrow{\phi} G$$ is trivial. By the exactness there is a unique map $\varphi:\pi^I({k}/k,\bar{x})\to G$ making the  diagram $$\xymatrix{\pi^I({X}/k,\bar{x})\ar[rr]\ar[dr]_-{\phi}&&\pi^I({k}/k,\bar{x})\ar@{.>}[ld]^{\varphi}\\&G&}$$ commutative. Therefore, $\varphi$ defines an object in $I(k/k,\bar{x})$ whose pull-back is isomorphic to $(P,G,p)$.

"$\Leftarrow$" Let $(P,G,p)\in I(X/k,\bar{x})$ be an $I$-saturated object. By \ref{imp}, it corresponds to a $k$-homomorphism $\phi:\pi^I({X}/k,\bar{x})\twoheadrightarrow G$. Let $H$ be the image of the composition $$\pi^I(\bar{X}/k,\bar{x})\to\pi^I(X/k,\bar{x})\xrightarrow{\phi} G.$$ By (i), $H\subseteq G$ is a normal subgroup. Thus we get an object $(P/H,G/H,p)\in I(X/k,\bar{x})$. Since the composition $$\pi^I(\bar{X}/k,\bar{x})\to\pi^I(X/k,\bar{x})\xrightarrow{} G\to G/H$$ is trivial by definition, the pull-back of $(P/H,G/H,p)$ to $\bar{X}$ is a trivial object. By (ii), $(P/H,G/H,p)$ descends to an object in $I(k/k,\bar{x})$, or equivalently, there is a commutative diagram $$\xymatrix{\pi^I(\bar{X}/k,\bar{x})\ar[r]\ar[d]&\pi^I({X}/k,\bar{x})\ar@{.>}[d]\\G\ar[r]&G/H}.$$
Let $N$ be the image of the kernel of $\pi^I({X}/k,\bar{x})\to \pi^I(k/k,\bar{x})$  under the map $${\phi}:\pi^I({X}/k,\bar{x})\twoheadrightarrow G.$$ The above diagram implies that $N\subseteq H$ and the first statement of this proposition implies that $H\subseteq N$. Therefore $H=N$. But since this is valid for all $I$-saturated objects, we can conclude the middle exactness.
\end{proof}

\begin{rmk} The sequence is in general not exact on the left. See \ref{infconex} for an example when $k$ is perfect $X=\A_k^1$ and $I=L$. The example does not work for $\pi^E$. However, one can not use the injectivity for $\pi^G$ (or $\pi_1^{\text{\'et}}$) to conclude the injectivity for $\pi^E$ either. The injectivity for $\pi^G$ (or $\pi_1^{\text{\'et}}$) was deduced from the theory of \textit{Galois closure} for Galois coverings \cite[Proposition 5.3.9, pp. 169]{Sz}. But we can not find an analogue for $\pi^E$.
\end{rmk}

\begin{cor} \label{AP}  If either $X=\A_{{k}}^n$ with $k$ is a field of characteristic 0 \textit{or} $X$ is a complete rational variety over an arbitrary field $k$, then the  canonical map $$\pi^N({X}/k,\bar{x})\to\pi^N(k/k,\bar{x})$$ is an isomorphism.
\end{cor}
\begin{proof} By \ref{projectivelemma} and \ref{affinelemma}, $\pi^N(\bar{X}/k,\bar{x})=\{1\}$. Then the corollary follows from \ref{exactness} and \ref{descentprop}.
\end{proof}

\begin{ex}\label{faildescent} In this part we would like to give an example to show that the condition (ii) of \ref{exactness} is not always satisfied.

Let's just take $k=\F_p(s,t)$ (the function field in two variables over $\F_p$),  $X=\A_k^1\setminus\{a\}$, and $$P=\Spec(A[T]/(T^p-(s+tx^p)))$$ be the $\mu_p$-torsor over $X$ in a natural way, where $A:=O_X(X)$ and $a\in \A_k^1$ is the closed point determined by the polynomial $s+tx^p\in k[x]$. Since $P$ is a  local torsor the base point plays no role. For this reason  we are going to omit the base point in the following discussion. Clearly, the equation $$T^p-(s+tx^p)=0$$ has no solution in $A$, thus $P$ is a non-trivial $\mu_p$-torsor. Furthermore, $P\times_k\bar{k}$ is a tivial torsor over $\bar{X}$, the section being given by the solution of the above equation in $\bar{k}[x]$. But $P\to X$ can not descent to a $\mu_p$-torsor over $\Spec(k)$. In fact, the two $\mu_p$-torsors
$$P_0=\Spec(k[T]/(T^p-s))\hspace{20pt}\text{and}\hspace{20pt}P_1=\Spec(k[T]/(T^p-s-t))$$
which are fibres of $P\to X$ at $x=0$ and $x=1$ respectively, can not be isomorphic. Suppose there was an isomorphism of torsors $$f:k[T]/(T^p-s)\longrightarrow  k[T]/(T^p-s-t)$$ sending $T\mapsto f(T)$, where $f(T)\in k[T]$ is a polynomial of degree less than $p$.  Let $\mu_p=k[Y]/(Y^p-1)$, then $\Aut_{k-\text{grp.sch}}(\mu_p)=(\Z/p\Z)^*$, where $m\in(\Z/p\Z)^*$ stands for $Y\mapsto Y^m$. Thus we should have a commutative diagram $$\xymatrix{
k[T]/(T^p-s)\ar[r]^f\ar[d]_{\rho_0}&k[T]/(T^p-s-t)\ar[d]^{\rho_1}\\k[T]/(T^p-s)\otimes_kk[Y]/(Y^p-1)\ar[r]^-{f\otimes m}&k[T]/(T^p-s-t)\otimes_kk[Y]/(Y^p-1)}, $$where $\rho_0$ and $\rho_1$ are defined by the action of $\mu_p$ on $P_0$ and $P_1$ respectively. Tracing the image of $T$ in the above diagram,  we get
$f(T)\otimes Y^m=f(T\otimes Y)$. This implies that $f(T)$ is a polynomial of the form $\alpha T^m$ with $\alpha\in k$. Then we should have $$f:\ \ \ T^p-s\longmapsto f(T)^p-s=\alpha^pT^{mp} -s=0\in k[T]/(T^p-s-t).$$
But $T^p=s+t\in k[T]/(T^p-s-t)$. Hence we should have $\alpha^p(s+t)^m-s=0\in k\subset k[T]/(T^p-s-t)$. However, this equation can not happen in $k$ because $\F_p[s,t]$ is a UFD.
\end{ex}

However, under some conditions (ii) actually holds.

\begin{prop}\label{descentprop} If  in \ref{exactness}, one of the following conditions hold,
\begin{compactitem}
\item the field $k$ is perfect and $X$ is in addition quasi-compact;
\item the scheme $X$ is proper;
\item the group $G$ is \'etale,
\end{compactitem}
then condition {\rm (ii)} holds for $N(X/k,\bar{x})$.
\end{prop}

\begin{proof}  The last case will be proved in \ref{descentlem}.  Let's show the first two. Let  $(P,G,p)\in N(X/k,\bar{x})$ be an object whose pull-back to $N(\bar{X}/k,\bar{x})$ is trivial, i.e. there is a morphism
$$\lambda:(\bar{X},\{1\},\bar{x})\to (\bar{P},G,p)\in N(\bar{X}/k,\bar{x}).$$

\textit{First assume that $k$ is perfect and $X$ is  quasi-compact.}  
Let $(P_{\text{\'et}},G_{\text{\'et}},p)\in I_{\text{\'et}}(X/k,\bar{x})$ be the \'etale quotient of $(P,G,p)$. Then $(\bar{P}_{\text{\'et}},G_{\text{\'et}},p)$ is also trivial. Thus by \ref{descentlem} there is a triple $(Q,H,q)\in I_{\text{\'et}}(k/k,\bar{x})\subseteq N(k/k,\bar{x})$ such that the pull-back of $(Q,H,q)$ to $X$ is isomorphic to $(P_{\text{\'et}},G_{\text{\'et}},p)$. Let $n $ be the order of the $k$-group scheme $G_{\text{\'et}}$. Then $\bar{P}_{\text{\'et}}$ can be written as $n$-copies of $\bar{X}$: $$\bar{P}_{\text{\'et}}=\coprod_{i=1,\cdots,n}\bar{X}_i$$ where $\bar{X}_i=\bar{X}$. The quotient  $\pi: P\to P_{\text{\'et}}$ makes $P$ as $G^0$-torsor over $P_{\text{\'et}}$, and we have a decomposition $$\bar{P}=\coprod_{i=1,\cdots,n}\bar{P}_i$$ where $\bar{P}_i$ is just  $(\pi\times_k\bar{k})^{-1}(\bar{X}_i)$. Suppose $p\in \bar{P}_1(\bar{k})$. Then the map $\lambda$ makes $\bar{P}_1$  a trivial $G^0$-torsor over $\bar{X}_1$. Since $G^0$ is local and $\bar{X}_1$ is reduced,  the closed imbedding $(\bar{P}_1)_{\text{red}}\hookrightarrow \bar{P}_1$ composing with the projection $\bar{P}_1\to\bar{X}_1$ has to be an isomorphism. As $G_{\text{\'et}}(\bar{k})$ acts transitively on the components $\bar{X}_i$,  for any $1\leq i\leq n$ there is an element $g\in G_{\text{\'et}}(\bar{k})=G(\bar{k})$ making the diagram $$\xymatrix{\bar{P}_0\ar[r]^g\ar[d]&\bar{P}_i\ar[d]\\ \bar{X}_0\ar[r]^g&\bar{X}_i}$$ commutative. Hence the closed imbedding $(\bar{P}_i)_{\text{red}}\hookrightarrow \bar{P}_i$ composing with the projection $\bar{P}_i\to\bar{X}_i$ is also an isomorphism for each $i$. Therefore the composition $\bar{P}_{\text{red}}\hookrightarrow \bar{P}\to \bar{P}_{\text{\'et}}$ is an isomorphism. But as $k$ is perfect, $\bar{P}_{\text{red}}=P_{\text{red}}\times_k\bar{k}$. Thus the composition $P_{\text{red}}\hookrightarrow P\to P_{\text{\'et}}$ has to be an isomorphism too. This defines a section $s:P_{\text{\'et}}\hookrightarrow P$ for the projection $\pi:P\to P_{\text{\'et}}$.  The universality of the reduced closed subscheme structure $P_{\text{red}}\subseteq P$ tells us that there is a unique arrow $P_{\text{\'et}}\times_k G_{\text{\'et}}\dashrightarrow P_{\text{\'et}}$ making the following diagram $$\xymatrix{P_{\text{\'et}}\times_kG_{\text{\'et}}\ar@{.>}[r]\ar@{>->}[d]^{s\times i}\ar@{=}@/_2pc/[dd]&P_{\text{\'et}}\ar@{>->}[d]^s\ar@{=}@/^2pc/[dd]\\P\times_kG\ar[r]^-{\rho}\ar@{>>}[d]^{\pi\times o}& P\ar@{>>}[d]^-{\pi}\\ P_{{\text{\'et}}}\times_kG_{\text{\'et}}\ar[r]^-{\rho_{\text{\'et}}}& P_{\text{\'et}}}$$ commutative, where $i:G_{\text{\'et}}\subseteq G$ is the inclusion of the reduced subscheme structure of $G$, $\rho_{\text{\'et}}$ is  action of $G_{\text{\'et}}$ on $P_{\text{\'et}}$ induced by $\rho$, $o:G\twoheadrightarrow G_{\text{\'et}}$ is the \'etale quotient map.  Therefore we obtain a morphism $$(P_{\text{\'et}},G_{\text{\'et}}, p)\to (P,G,p)\in N(X/k,\bar{x}).$$In view of the isomorphism  $(P_{\text{\'et}},G_{\text{\'et}}, p)\cong (Q\times_kX,H,q)$, we can equip the  $k$-scheme $G$ with a left action from $H$ via $H\cong G_{{\text{\'et}}}\xrightarrow{i}G$, then the contracted product $Q\times^HG$  provides a $k$-form for the $G$-torsor $P$ over $X$. This finishes the proof the first case.

\textit{Now suppose that $X$ is proper.}  Let $f: X\to\Spec(k)$ be the structure morphism, and $\mathcal{A}$ be the push-forward of $\mathcal{O}_P$ to $X$ along $P\to X$. Then $\mathcal{A}$ is a locally free coherent $\mathcal{O}_X$-algebra equipped with a $G$-action map $$\rho: \mathcal{A}\longrightarrow \mathcal{A}\otimes_kk[G]$$ which makes the induced map $\mathcal{A}\otimes_{\mathcal{O}_X}\mathcal{A}\xrightarrow{\id\otimes \rho} \mathcal{A}\otimes_kk[G]$ into an isomorphism. Since $\bar{P}$ is a trivial $G$-torsor over $\bar{X}$, 
$\bar{\mathcal{A}}:=\mathcal{A}\otimes_k\bar{k}$ is a free $\mathcal{O}_{\bar{X}}$-module. But $X$ is proper separable and geometrically connected over $k$, so the adjunction map  $\bar{f}^*\bar{f}_*\bar{\mathcal{A}}\to\bar{\mathcal{A}}$ is an isomorphism. By FPQC-descent, we have that $f^*f_*\mathcal{A}\to\mathcal{A}$ is an isomorphism.  Let $A:=f_*\mathcal{A}=\Gamma(X,\mathcal{O}_X)$. Then the action $f_*\rho: A\to A\otimes_kk[G]$ makes $\Spec(A)$ into a $G$-torsor whose pull-back to $X$ is precisely $P=\Spec(\mathcal{A})$.
\end{proof}


 \subsection{The \'Etale Case}

\begin{prop}\label{topthm} Let $X$ be a  Noetherian scheme, which is  geometrically connected
over a  field $k$, and let $\bar{x}:\Spec(\bar{k})\hookrightarrow X$,  $(P,G,p)\in I_{\text{\'et}}(X/k,\bar{x})$ be a saturated object. Let ${N}$ be the image of the following composition $$\pi^{E}(\bar{X}/k,\bar{x})\to\pi^E(X/k,\bar{x})\twoheadrightarrow G.$$ Then we get an imbedding $(\bar{P}',N,p)\subseteq (\bar{P},G,p)\in I_{\text{\'et}}(\bar{X}/k,\bar{x}).$ If one of the following conditions is satisfied, then $N\subseteq G$ is a normal subgroup scheme.
\begin{enumerate} [{\rm(i)}]\item $P$, as a scheme, is connected.
\item $\bar{P}'$, as a scheme, is connected.
\end{enumerate}
\end{prop}
\begin{proof} By \cite[\'Expos\'e IX, Th\'eor\`eme 4.10]{SGA1} we may assume that $k$ is a perfect field. There is a finite Galois subextenison $k\subseteq K$ of $k\subseteq \bar{k}$ such that $G_K$ is constant over $K$ and the number of connected components of $P_K$ is the same as that of $\bar{P}$. In this case the image of $\pi^{E}(\bar{X}/k,\bar{x})$ and $\pi^{E}({X}_K/k,\bar{x})$ are the same in $G$.
Thus replacing $\bar{k}$ by $K$,  we may assume that $k\subseteq \bar{k}$ is a finite Galois extension.

Suppose that we have an $I_{\text{\'et}}$-saturated object $(P,G,p)\in I_{\text{\'et}}(X/k,\bar{x})$. Let $\bar{G}:=G\times_k\bar{k}$, $\bar{P}:=P\times_k\bar{k}$, $\bar{P}_0$ be the
connected component of $\bar{P}$ containing $p$. Let $\bar{H}\subseteq\bar{G}$ be the subgroup which fixes $\bar{P}_0$, i.e.$$\bar{H}:=\{\ g\in\bar{G}\ |\ \bar{P}_0\ g=\bar{P}_0\ \}.$$
Then $\bar{N}\subseteq\bar{G}$ is the smallest subgroup which contains the subset
$$\bigcup_{\sigma\in \Gal(\bar{k}/k)}\sigma(\bar{H})\hspace{15pt}\subseteq\  \bar{G}.$$

Now let $T$ be the subset of $\bar{G}$ whose elements are those $t_{\sigma}$
which send $p\mapsto\sigma(p)$ for some $\sigma\in \Gal(\bar{k}/k)$.
Let $\bar{M}$ be the smallest subgroup of $\bar{G}$ containing $T$. Since
for any $\sigma\in\Gal(\bar{k}/k)$ and $t_{\tau}\in T$ sending $p\mapsto
\tau(p)$ we have $ t_{\sigma}\circ\sigma(t_{\tau})=t_{\sigma\tau}$. So $\sigma(t_{\tau})=t_{\sigma}^{-1}\circ t_{\sigma\tau}\in \bar{M}$, then it follows that
$\sigma(\bar{M})\subseteq \bar{M}$. Let $\bar{M}\bar{N}\subseteq \bar{G}$ be the smallest subgroup of $\bar{G}$ containing $\bar{M}$ and $\bar{N}$.
 Let $$\bar{P}''\ :=\ \bigcup_{g\in\bar{M}\bar{N}} (\bar{P}_0)g\hspace{15pt}\subseteq \ \bar{P},$$  i.e. the $\bar{M}\bar{N}$-obits of $\bar{P}_0$ in $\bar{P}$. Then $\bar{P}''$ is a torsor
under $\bar{M}\bar{N}$. Since both $\bar{P}''$ and $\bar{M}\bar{N}$ are stable under the induced Galois
action, they both descend to $k$, i.e. there exists a $k$-form $MN\subseteq G$ of $\bar{M}\bar{N}$ such that $\bar{P}''$ descends to an $MN$-torsor $P''\subseteq P$ over $X$. Then there is a homomorphism
$$\pi^E(X/k,\bar{x})\to MN\subseteq G.$$ But $\pi^E(X/k,\bar{x})\to G$ is already surjective by the assumption, this immediately implies that $MN=G$ or equivalently $\bar{M}\bar{N}=\bar{G}$.

Next we show that $\bar{N}$ is a normal subgroup of $\bar{G}$. From
the above discussion it is enough to check
$m^{-1}\bar{H}m\subseteq\bar{N}$ for $\forall m\in \bar{M}$. If (i) is satisfied, then $\Gal(\bar{k}/k)$ acts transitively on the connected components of $\bar{P}$, so any element $g\in \bar{G}$ can be written as $h\circ t_{\sigma}$, where $h\in\bar{H}$, $t_{\sigma}\in T$. If (ii) is satisfied, then $\bar{H}=\bar{N}$. In either case it is
already enough to check $t_{\sigma}^{-1}\bar{H}t_{\sigma}\subseteq\bar{N}$ for
 $\forall t_{\sigma}\in T$. From the very definition of $t_{\sigma}$ we have $$\sigma(p) t_{\sigma}^{-1}ht_{\sigma}= (ph)t_{\sigma} =\sigma(p)h'$$ where $h'$ is contained in the stabilizer of $\sigma(P_0)$, i.e. $\sigma(H)$. Thus $t_{\sigma}ht_{\sigma}^{-1} \in\sigma(H)\subseteq\bar{N}$.
\end{proof}

\begin{rmk} The two conditions in \ref{topthm} are already satisfied by most interesting \'etale torsors. For example it is known in the literature \ref{conn} that any triple $(P,G,p)\in I_{\text{co}}(X/k,\bar{x})$ is $I_{\text{co}}$-saturated if and only if it is connected.  Thus in view of \ref{descentlem} and \ref{exactness} this proposition  can be seen as a generalization of the
fundamental exact sequence of the \'etale fundamental group
\cite[Expos\'e IX, Th\'eor\`eme 6.1]{SGA1}. 
\end{rmk}

\begin{ex}\label{normality} In this part, we would like to construct an example showing that for a saturated object $(P,G,p)\in I_{\text{\'et}}(X/k,\bar{x})$ which does not satisfy any of the conditions in \ref{topthm}  the image of the composition $$\pi^{E}(\bar{X}/k,\bar{x})\to\pi^E(X/k,\bar{x})\twoheadrightarrow G$$ needs  not  be a normal subgroup of $G$.

Let $X$ be any scheme which is geometrically connected over a field $k$, and let $Y\to X$ be a torsor under the constant group scheme $(\Z/2\Z)_k$ with $Y$ geometrically connected over $k$.  Now we are going to construct a finite Galois field extention $K/k$ with Galois group $M$,  a torsor $P_K'$ over $X_K$ under an abstract group $G_K'$ which contains the $X_K$-scheme $Y_K$  as a connected component. We will also construct a twisted action of $M$ on the $X_K$-scheme $P_K'$ and an action of $M$ on the abstract group $G_K'$ in such a way that these two actions are compatible.

Let $n\in\N^+$ be an even number which is equal to or larger than 2. Let $$G_K'=((\langle a_1\rangle\times\langle a_2\rangle\times\langle a_3\rangle\times\langle a_4\rangle)\rtimes(\langle b_1\rangle\times\langle b_2\rangle))\rtimes\langle \xi\rangle,$$where $\langle a_1\rangle=\langle a_2\rangle=\langle a_3\rangle=\langle a_4\rangle\cong\Z/2n\Z$, $\langle b_1\rangle=\langle b_2\rangle=\langle \xi\rangle\cong\Z/2\Z$. The actions are defined by the following relations.

\[\xymatrix{b_1a_1=a_2b_1& b_1a_2=a_1b_1&b_1a_3=a_3b_1&b_1a_4=a_4b_1}
\]
\[\xymatrix{ b_2a_1=a_1b_2& b_2a_2=a_2b_2& b_2a_3=a_4b_2&b_2a_4=a_3b_2}
\]
\[\xymatrix{ \ \ \ \xi b_1=b_1\xi &\xi b_2=a_3^na_4^nb_2\xi&\xi a_i=a_i^{n+1}\xi &i=1,2,3,4.}
\]

In addition we can define an action of $\Z/2\Z=\{e,\sigma\}$ on $G_K'$ (via group automorphisms). The action is given by the following equations.
\[\xymatrix{\sigma(a_1)=a_3&\sigma(a_2)=a_4&\sigma(a_3)=a_1&\sigma(a_4)=a_2}
\]
\[\xymatrix{\sigma(b_1)=b_2&\sigma(b_2)=b_1&\sigma(\xi)=a_1^na_3^n\xi.&\hspace{40pt}}
\]

Next we construct the $G_K'$-torsor $P_K'$. Let $H_K'\subseteq G_K'$ be the subgroup $$(\langle a_1\rangle\times\langle a_2\rangle\times\langle a_3\rangle\times\langle a_4\rangle)\rtimes(\langle b_1\rangle\times\langle b_2\rangle)\subseteq G_K'$$ and let $$P':= \coprod_{i\in H_K'} Y_i$$ be the disjoint union of copies of $Y$ (i.e. $Y_i=Y$). We define a right action of $G_K'$ on $P'$  in the following way. If $j\in H_K'$ then the action of $j$ on $Y_i$ is defined by the identity morphism $Y=Y_i\to Y_{ij}=Y$. If $j\notin H_K'$, then $ij$ is uniquely written as a product $ij=\xi k$ with $k\in H_K'$. Then the action of $j$ on $Y_i$ will be the morphism  $Y=Y_i\to Y_k=Y$ given by the action of the non-trivial element of $\Z/2\Z$ (remember that $Y$ is a $\Z/2\Z$-torsor over $X$). Viewing $G_K'$ as a constant group scheme over $k$, one gets a morphism $$\rho: P'\times_k G_K'\to P'$$ over $X$ defining the right action. This action actually defines  $P'$ as a $G_K'$-torsor over $X$. Indeed, one can take a geometric point $\bar{x}\in X$. Then the fibre of $\bar{x}$ under the projection $Y_e=Y\to X$ consists of two points. We pick any point in the fibre and denote it by $p$. Then the other point is $p\xi$. We can also translate $p$ and $p\xi$ by the group $H_K'$. In this way we get all the fibres of $\bar{x}$ under the projection $P'\to X$. Each fibre can be written uniquely as $pi$ or $p\xi i$ for some $i\in H_K'$. By the very definition of the action of $G_K'$ on the set of fibres of $\bar{x}$, one sees that the set of all fibres of $\bar{x}$ is a principal homogeneous space under $G_K'$. Hence the  $X$-morphism $$P'\times_kG_K'\xrightarrow{id\times\rho} P'\times_{X}P'$$ induces an isomorphism at the fibre of $\bar{x}\in X$. Since $P'\times_kG_K'$ and $P'\times_{X}P'$ are all finite \'etale $X$-schemes and $X$ is connected, the morphism $id\times\rho$ is an isomorphism by \cite[Expos\'e V, Th\'eor\`eme 4.1]{SGA1}. Therefore $P'$ is a $G_K'$-torsor over $X$. Moreover, we would like to introduce two actions on $P'$ by $\Z/2n\Z=\langle u\rangle$ and $\Z/2\Z=\langle v\rangle$ respectively. The action of $u$ on a component $Y_i$ is defined to be the identity morphism $Y=Y_i\to Y_{a_1a_3\sigma(i)}=Y$. Notice that we have a commutative diagram \begin{equation}\label{skull}\xymatrix{P'\ar[r]^-g\ar[d]_-u&P'\ar[d]^-u\\P'\ar[r]^-{\sigma(g)}&P'}\tag{$\skull$}\end{equation} for all $g\in G_K'$, even when $g\notin H_K'$. The action of $v$ on a component  $Y_i$ is defined to be the identity morphism $Y=Y_i\to Y_{b_1i}=Y$. Similarly, we have a commutative diagram  \begin{equation}\label{queen}\xymatrix{P'\ar[r]^-g\ar[d]_-v&P'\ar[d]^-v\\P'\ar[r]^-{g}&P'}\tag{\symqueen}\end{equation} for all $g\in G_K'$.

Next we would like to construct a finite group $M$ generated by two elements $\{x,y\}$ and two group homomorphisms $$f: M\rightarrow\Aut_{X}(P') \hspace{20pt}\text{and}\hspace{20pt} g: M\rightarrow \Aut_{}(G_K')$$ such that $f(x)=u,f(y)=v$ and $g(x)=\sigma, g(y)=id$, where $P'$ is considered as an object in the category of $X$-schemes and $G_K'$ is considered as an object in the category of abstract groups. There should be some general procedure to obtain such $M$ and $f,g$, because all the automorphism groups are finite. But unfortunately the author has to rely on some brutal computational methods. For  simplicity   we treat only the case when $n=2$.
In this case, we first consider the following equations.
\begin{enumerate}[(i)]
\item $$xyx^2yx=yxyx^2yxy=yx^3yx^2yx^3y$$
\item $$yx^2yx^2=x^2yx^2y$$
\item $$\begin{aligned} &x^3yxy=yx^3yx\\&yxyx^3=xyx^3y\end{aligned}$$
\item $$\begin{aligned} &x^3yx^3yx^2yxy=yx^3yx^3yx^3y=yxyx^2yx^3yx^3\\& xyxyx^2yx^3y=yxyxyxy=yx^3yx^2yxyx\end{aligned}$$
\item $$x^3yx^2yxyx^2y=xyx^2yx^3yx^2y=yx^2yx^3yx^2yx=yx^2yxyx^2yx^3$$
\item $$yxyxyxyx=x^3yx^3yx^3yx^3y$$
\item $$x^4=1\hspace{50pt} y^2=1$$
\end{enumerate}

One observes first that the above equations hold in $ \Aut_{}(G_K')$ when one replaces $x$ by $\sigma$, $y$ by $id$, and the same hold  in $\Aut_{X}(P')$ when one replaces $x$ by $u$, $y$ by $v$. The former is somewhat clear, the latter needs some computation. To verify the above equations for $u,v$, we choose a geometric point $\bar{x}\in X$ and a fibre $p$ of $\bar{x}$ under $Y_e=Y\to X$, then check whether the actions from both sides of the equation are agree on $p$. If so, one could then use \ref{queen}, \ref{skull} and the compatibility in $ \Aut_{}(G_K')$ to move $p$ around, through all the fibres of $\bar{x}$ under $P_K'\to X_K$, and finally conclude that the actions from both sides of the equation  agree on all fibres. Then the equations for $u,v$ follow from \cite[Expos\'e V, Th\'eor\`eme 4.1]{SGA1}. Here is the result of the calculations.
\begin{enumerate}[(i)]
\item $$uvu^2vu(p)=vuvu^2vuv(p)=vu^3vu^2vu^3v(p)=p(a_3a_4)^2$$
\item $$vu^2vu^2(p)=u^2vu^2v(p)=p(a_1a_2)^2$$
\item $$\begin{aligned} &u^3vuv=vu^3vu(p)=pa_3^3a_4b_1b_2\\&vuvu^3(p)=uvu^3v(p)=pa_3a_4^3b_1b_2\end{aligned}$$
\item $$\begin{aligned} &u^3vu^3vu^2vuv(p)=vu^3vu^3vu^3v(p)=vuvu^2vu^3vu^3(p)=pa_1^3a_3^3a_2^2a_4^2\\& uvuvu^2vu^3v(p)=vuvuvuv(p)=vu^3vu^2vuvu(p)=pa_1a_3a_2^2a_4^2\end{aligned}$$
\item $$vuvuvuvu(p)=u^3vu^3vu^3vu^3v(p)=p(a_1a_2a_3a_4)^2$$
\item $$\begin{aligned}&u^3vu^2vuvu^2v(p)=uvu^2vu^3vu^2v(p)\\=&vu^2vu^3vu^2vu(p)=vu^2vuvu^2vu^3(p)\\=&p(a_1a_2a_3a_4)^2\end{aligned}$$
\item $$u^4(p)=p\hspace{50pt} v^2(p)=p$$\end{enumerate}

Let $M$ be the free group generated by $x,y$ modulo the relations (i)-(vii). One can see without too much difficulty that $M$ is a finite group generated by $x,y$. Clearly there are group homomorphisms $f: x\mapsto u, y\mapsto v$ and $g: x\mapsto \sigma, y\mapsto id$.

Now let $L$ be any field of any characteristic. Choose an imbedding $M\subseteq S_m$ for some $m\in \N^+$, we get a faithful action of $M$ on $L(X_1,X_2,\cdots,X_m)$. Let $$K:=L(X_1,X_2,\cdots,X_m)\hspace{50pt} k:=K^M$$ where $K^M\subseteq K$ denotes the subfield of invariant elements under the action of $M$. Then $K/k$ is a finite Galois extension with Galois group $M$. 

Let $P_K':=P'\times_kK$. Then $P_K'$ is a $G_K'$-torsor over $X_K$. We also have an imbedding $Y_K=Y_e\times_kK\subseteq P'\times_kK=P_K'$. Since the connected covering $Y_K\to X_K$ comes, via base change,  from a Galois covering $Y\to X$ over $k$, it has to be again a Galois covering \cite[Expos\'e V, \S4, f), (ii)]{SGA1}. Thus the inclusion $\Aut_{X}(Y)\subseteq \Aut_{X_K}(Y_K)$ has to be an isomorphism (because $Y/X$ and $Y_K/X_K$ are of the same degree).

Now for each element $\alpha\in M$ we can define a twisted action on $P_K'$ via $$\xymatrix{ P_K'=P'\times_kK\ar[rr]^-{f(\alpha)\times\alpha}&& P'\times_kK =P_K'}.$$ By \ref{queen} and \ref{skull} this twisted action is compatible with the action of $M$ on $G_K'$.

Viewing $G_K'$ as a constant group scheme over $K$ and applying Galois descent we get a $k$-group scheme $G_k'$ and a right $G_k'$-torsor $P_k'$ over $X$ such that the pull-back of $(P_k',G_k')$ to $K$ is $(P_K', G_K')$. Choosing a geometric point $\bar{x}: \Spec(\bar{K})\to X_K$ and a lifting $p:\Spec(\bar{K})\to Y_e\times_kK\subseteq P_K'$, we get a triple $(P_k',G_k',p)\in I_{\text{\'et}}(X/k,\bar{x})$. This triple corresponds to a homomorphism $\pi^{E}(X/k,\bar{x})\to G_k'$. Let $G\subseteq G_k'$ be the image, $(P,G,p)\subseteq (P_k',G_k',p)$ be the triple in $I_{\text{\'et}}(X/k,\bar{x})$ corresponding to  $\pi^{E}(X/k,\bar{x})\twoheadrightarrow G\subseteq G_k'$.

In this case, $(P,G,p)$ is a saturated object by definition, and the pull-back $G_K\subseteq G_K'$ is a subgroup stable under the action of $M$. Since $P_K\subseteq P_K'$ is a subscheme containing $p\in Y_e\times_kK$ and is stable under the action of $M$, $P_K$ contains $$Y_e\times_kK,\ \ \ \  Y_{b_1}\times_kK=y(Y_e\times_kK),\ \ \ \ Y_{a_1a_3}\times_kK=x(Y_e\times_kK).$$ As $Y_{b_1}=Y_e b_1$ and $Y_{a_1a_3}=Y_ea_1a_3$, $G_K$ contains $\xi, b_1, a_1a_3$. Just like in   \ref{topthm}, we denote by $N$ the image of the homomorphism $$\pi^N(\bar{X}/k,\bar{x})\to  G$$   corresponding to the triple $(\bar{P}, {G},p)$. Then $\bar{N}$ is generated by the $\Gal(\bar{k}/k)$-orbit of $\{e,\xi\}\subseteq G_K=\bar{G}$, or equivalently by the $M$-orbit of $\{e,\xi\}\subseteq {G}_K$, i.e.   $$\bar{N}:=\{e,\xi,(a_1a_3)^n, (a_1a_3)^n\xi\}.$$ However, $\bar{N}\subseteq \bar{G}$ is not a normal subgroup for $b_1(a_1a_3)^nb_1^{-1}=(a_2a_3)^n\notin \bar{N}$.
\end{ex}
\begin{rmk} (i) In the above example, we could take $X:=\A_k^1$ and $Y\to X$ to be the Artin-Schreier covering under $\F_2$ if $k$ is of characteristic 2. \\
(ii) We could also take $A$ to be an abelian variety over a field $l$ of characteristic $\neq 2$, then $D:=A[2]\times_l\bar{l}$ is a constant group scheme of order $2^{2\dim(A)}$.  Suppose $A[2]$ stays as a constant group scheme over $l\subseteq L\subseteq \bar{l}$ and $D\twoheadrightarrow \Z/2\Z$ is a surjective homomorphism. Let $k:=L(X_1,X_2,\cdots,X_m)^M$ be as in the example, $X:=A\times_lk$. Then we could define the $\Z/2\Z$-torsor $Y\to X$ to be the one obtained by taking the contracted product of the $D$-torsor $X\xrightarrow{2\cdot} X$ along $D\twoheadrightarrow\Z/2\Z$. Clearly $Y$ is  geometrically connected over $k$. In this example $k$ is allowed to be of  any characteristic  $\neq 2$.
\end{rmk}

\begin{lem}\label{descentlem} Let $X$ be a  geometrically connected quasi-compact scheme over a  field $k$, $G$ be a finite \'etale group scheme over $k$, and $P$ be a $G$-torsor  over $X$.  If $\bar{P}$ is a trivial $G$-torsor over $\bar{X}$, then there is a $G$-torsor $Q$ over $k$ whose pull-back along $X\to k$ is  $P$.
\end{lem}

\begin{proof} Let $K$ be an intermediate finite Galois extension of $k\subseteq \bar{k}$ over which $G_K$ becomes a constant group scheme and $P_K$ remains a trivial torsor. Since $P_K$  is a trivial $G$-torsor over $X_K$, by choosing an $X_K$-section for the projection $\pi:P_K\to X_K$ we get  isomorphisms (in the category of $X_K$-schemes) $${P}_K\cong{X}_K\times_kG={X}_K\times_{{K}}G_K=\coprod_{i\in{G}_K}{X}_K.$$  By Galois descent, giving the $X$-scheme $P$ is equivalent to giving a twisted action of $\Gal(K/k)$ on ${P}_K$, i.e. a   homomorphism $$f:\Gal( {K}/k)\to \Aut_X({P}_K)$$ such that the following diagram\begin{equation}\label{semilinear}\xymatrix{P_K\cong\coprod_{i\in{G}_K}{X}_K\ar[rr]^-{f(\sigma)}\ar[d]_{\pi}&&\coprod_{i\in{G}_K}{X}_K\cong P_K\ar[d]^{\pi}\\  {X}_K\ar[rr]^{id\times\sigma}\ar[d]&& {X}_K\ar[d]\\ K
\ar[rr]^{\sigma}&&K
}\tag{$*$}\end{equation} is commutative for all $\sigma\in\Gal(K/k)$. One observes that, as $X_K$ is connected, such a twisted action on $P_K\cong\coprod_{i\in{G}_K}{X}_K$ is none other than a permutation of the connected components in a twisted manner. The observation can be written  more formally.

Let $n\in\N^+$ be the order of the $k$-group scheme $G$, $S_n$ be the $n$-th permutation group. Then there is a unique group homomorphism $$\lambda_X: S_n\times \Gal(K/k)\to \Aut_X(P_K)$$  whose restriction to $S_n$ is the permutation of the connected components of $P_K$ and whose restriction to $\Gal(K/k)$ is  $$\hspace{20pt}\sigma\ \  \longmapsto\ \ P_K\cong(\coprod_{i\in G_K}X)\times_kK\xrightarrow{id\times \sigma}(\coprod_{i\in G_K}X)\times_kK\cong P_K. $$ The observation means that there is a group homomorphism $\theta: \Gal(K/k)\to S_n$ making the following diagram $$\xymatrix{\Gal(K/k)\ar[rr]^-{f}\ar[dr]_-{\theta\times id}&&\Aut_X(P_K)\\&S_n\times \Gal(K/k)\ar[ur]_-{\lambda_X}&}$$ commutative.

In the above, we could replace $P_K$ by the $k$-scheme $G_K=\coprod_{i\in{G}_K}\Spec(K)$ to obtain a homomorphism $$\lambda_k: S_n\times \Gal(K/k)\to \Aut_k(G_K)$$ where $G_K$ is regarded as an object in the category of $k$-schemes. In this way, we get a homomorphism $$g:\Gal( {K}/k)\xrightarrow{\theta\times id}S_n\times\Gal(K/k)\xrightarrow{\lambda_k} \Aut_k({G}_K)$$ making the following diagram\[ \xymatrix{G_K=\coprod_{i\in{G}_K}\Spec(K)\ar[rr]^-{g(\sigma)}\ar[d] &&\coprod_{i\in{G}_K}\Spec(K)=G_K\ar[d] \\ K
\ar[rr]^{\sigma}&&K
} \] commutative for each $\sigma\in \Gal(K/k)$. In other words, we get a twisted action of $\Gal(K/k)$ on $G_K$. This defines a $k$-form $Q$ for the $K$-scheme $G_K$.

The $X$-scheme $X\times_kQ$ is an $X$-form of the $X_K$-scheme $P_K\cong X_K\times_KG_K=X\times_kG_K$. From the very definition of $Q$ we see that the twisted action of $\Gal(K/k)$ on  $P_K$ corresponding to the two $X$-forms $X\times_kQ$ and $P$ are the same. Therefore, by Galois descent $P\cong X\times_kQ$ as $X$-schemes. On the other hand  since $G$ is a $k$-form of the $K$-group scheme $G_K$, there is an action via group automorphisms $$\phi:\Gal(K/k)\to \Aut_{\text{grp}}(G_K)$$ corresponding to $G$. As $P$ is a $G$-torsor we have the following commutative diagram $$\xymatrix{P_K\ar[rr]^{a}\ar[d]_{f(\sigma)}&&P_K\ar[d]^{f(\sigma)}\\P_K\ar[rr]^{\phi(\sigma)(a)}&&P_K}$$ for all $\sigma\in\Gal(K/k)$ and $a\in G_K$. Since the identification $P_K\cong X_K\times_KG_K$ is equivariant under the right actions and the action of $G_K$  on $P_K$ is also just a permutation of connected components, we have another commutative diagram  $$\xymatrix{G_K\ar[rr]^{a}\ar[d]_{g(\sigma)}&&G_K\ar[d]^{g(\sigma)}\\G_K\ar[rr]^{\phi(\sigma)(a)}&&G_K}$$for all $\sigma\in\Gal(K/k)$ and $a\in G_K$. Therefore, by Galois descent there is an action of $G$ on $Q$ over $k$ which makes the isomorphism $P\cong X\times_kQ$ equivariant under $G$. Hence $Q$ is a  $G$-torsor over $k$ whose pull-back to $X$ is the $G$-torsor $P$.
\end{proof}

\subsection{The Infinitesimal Case}

\begin{prop}\label{localcase} Let $X$ be a geometrically connected separable scheme  over a field $k$.
Let $\bar{x}: \Spec(\bar{k})\to X$ be a geometric point. Then the canonical map$$\pi_k^{\bar{k}}: \pi^{L}(\bar{X}/k, \bar{x})\to\pi^{L}(X/k,\bar{x})$$ is surjective, but {\rm not}, in general, an isomorphism.
\end{prop}

\begin{proof} Suppose we have a saturated object  $(P, G,p)\in I_{\text{lc}}(X/k,\bar{x})$. We take the image of the composition $$\pi^{L}(\bar{X}/k,\bar{x})\to\pi^{L}(X/k,\bar{x})\to G,$$ and denote it by $H$. By \ref{imp} there is $(\bar{Q},H,{q})\in N(\bar{X}/k,\bar{x})$ with a morphism $$(\bar{Q},H,{q})\hookrightarrow (\bar{P},G,p).$$
As $H\subseteq G$ is an infinitesimal closed imbedding (closed imbedding with nilpotent ideal sheaf), so is $\bar{Q}\subseteq \bar{P}$. Now take the quotient by $H$ on both sides. We get a section: $$\bar{s}: \bar{X}\cong \bar{Q}/H\to \bar{P}/H.$$
Since the projection $\bar{P}\to\bar{P}/H$ is faithfully flat, the ideal sheaf of  the section $\bar{s}$ is contained in the ideal sheaf of the infinitesimal imbedding $\bar{Q}\subseteq \bar{P}$. Hence imbedding $\bar{s}:\bar{X}=\bar{Q}/H\hookrightarrow \bar{P}/H$ is also  infinitesimal. As $\bar{X}$ is reduced, $\bar{X}=(\bar{P}/H)_{\rm red}$ is the unique reduced closed subscheme of $\bar{P}/H$. Because $k$ is perfect, we have$${(\bar{P}/H)}_{\rm red}=(P/H)_{\rm red}\times_k\bar{k}\hookrightarrow(P/H)\times_k\bar{k}= \bar{P}/H.$$ We have known that the composition $$\bar{X}=(\bar{P}/H)_{\rm red}\hookrightarrow \bar{P}/H\to \bar{X}$$ is an isomorphism, so $(P/H)_{\rm red}\to X$ is also an isomorphism. In this way we get a section $s$ for the $X$-scheme $P/H$. Now we pull back the $H$-torsor $P\to P/H$ via $s$. $$\xymatrix{Q\ar[r]\ar[d]& X\ar[d]^s\\P\ar[r]& P/H}$$ Then we get a triple $(Q,H,q)\in I_{\text{lc}}(X/k,\bar{x})$   which dominates $(P,G,p)$. In other words, the map $\pi^L(X/k,\bar{x})\to G$ factors through the imbedding $H\subseteq G$. Hence if $(P,G,p)$ is saturated in $I_{\text{lc}}(X/k,\bar{x})$ then $(\bar{P},G,p)$ is saturated in $I_{\text{lc}}(\bar{X}/k,\bar{x})$. This means that $\pi_k^{\bar{k}}$ is surjective. For the failure of the injectivity see \ref{infconex}.
\end{proof}

\begin{cor} \label{loccor} Let $X$ be a geometrically connected separable scheme  over a field $k$. Let $\bar{x}: \Spec(\bar{k})\to X$ be a geometric point.
The canonical map$$\pi_k^{\bar{k}}: \pi^{L}(\bar{X}/k, \bar{x})\to\pi^{L}(X/k,\bar{x})$$ is an isomorphism if and only if for any $G$-torsor $Y\to \bar{X}$
with $G$ a finite local $k$-group scheme, there exists a  $G$-torsor
${P}$ over ${X}$ whose pull-back is isomorphic to $Y$  as a $G$-torsor.
\end{cor}
\begin{proof} This is an immediate consequence of \ref{imp} and \ref{localcase}.
\end{proof}

\begin{lem}\label{localdescent} Let $X$ be a scheme  over a perfect field $k$ of characteristic $p$. If there is a {\rm reduced} $X$-scheme $Y$ whose pull-back $\bar{Y}$ is a torsor over  $\bar{X}$ under an infinitesimal $k$-group scheme $G$, and if $Y'$ is an $X$-scheme, then any $\bar{X}$-isomorphism $\phi:\bar{Y'}\cong \bar{Y}$ descends to  $X$.
\end{lem}

\begin{proof} The claim is true if only if there is a map $\varphi: Y\to Y'$ fitting into the following diagram $$\xymatrix{\bar{Y}\ar[r]^{\phi}\ar[d]&\bar{Y'}\ar[d]\\Y\ar@{.>}[r]^{\varphi}& Y'}.$$ The problem being local on $X$, we could assume $X=\Spec(A)$, $Y=\Spec(B)$, $Y'=\Spec(B')$. We have to show that the image of the composition $\iota:B'\to B'\otimes_k\bar{k}\to B\otimes_k\bar{k}$ lands on $B$.

Since $\bar{Y}$ is a torsor over $\bar{X}$ under an infinitesimal group scheme, for any $x\in B\otimes_k\bar{k}$, $x^{p^n}\in A\otimes_k\bar{k}$ for $n\in \N$ sufficiently large. This implies that for any $x\in B$, $x^{p^n}\in A$ for $n\in \N$ sufficiently large, because $A\otimes_k\bar{k}\cap B=A$ inside $B\otimes_k\bar{k}$. Conversely, if $x\in B\otimes_k\bar{k}$ and $x^{p^n}\in A$ for some $n\in \N$, then $x\in B$. Indeed, as $k$ is perfect, we can assume $x\in B\otimes_kl$ for some finite separable extension $l/k$ of degree $m$. Let $l=k(\alpha)$ for some primitive element $\alpha\in l$. Then $x$ can be uniquely written as $x=s_0+s_1\otimes\alpha+s_2\otimes\alpha^2+\cdots+s_{m-1}\otimes\alpha^{m-1}$ with $s_i\in B$. Since $\alpha^{p^n}$ is still a primitive element in $l$, i.e. $l=k(\alpha)=k(\alpha^{p^n})$, $x^{p^n}\in A$ implies that $s_i^{p^n}=0$ for all $i>0$. As $B$ is reduced, $s_i=0$ for all $i>0$, hence $x\in B$. Thus $B\subseteq B\otimes_k\bar{k}$ is the subset consisting of elements whose $p^n$-th power is in $A$.

By the same argument as above, any element $x\in B'$ has $p^n$-th power in $A$. Hence $\iota(B')\subseteq B\otimes_k\bar{k}$ is contained in $B$. This completes the proof.
\end{proof}

\begin{cor}\label{infconex} Let $k$ be a perfect but not separably closed field of characteristic $p$.  If $X$ is $\A_{k}^1$  or an elliptic curve such that $X(F)=\alpha_{p,k}$ (where $F$ is the relative Frobenius),   then the surjective map $$\pi^L(\bar{X}/k,\bar{x})\to\pi^L(X/k,\bar{x})$$ is not an isomorphism. \end{cor}
\begin{proof} In any case, we have a non-trivial $\alpha_{p,k}$-torsor $F:X\to X$ defined by the relative Frobenius. Taking any $a\in \bar{k}\setminus k$ we can define a $\bar{k}-$automorphism of $\alpha_{p,\bar{k}}$ by the following map of Hopf-algebras:
$\bar{k}[x]/x^p\to \bar{k}[x]/x^p$ sending $x\mapsto ax$. This automorphism of $\alpha_{p,\bar{k}}$ defines a new action of $\alpha_{p,k}$ on $F:\bar{X}\to \bar{X}$ which makes $\bar{X}$ an $\alpha_{p,k}$-torsor over itself. But the new action $\bar{X}\times_k\alpha_{p,k}\to\bar{X}$ certainly does not descend to $X\times_k\alpha_{p,k}\to X$. However, if $\pi^L(\bar{X}/k,\bar{x})\to\pi^L(X/k,\bar{x})$ was an isomorphism, then by \ref {loccor} and \ref{localdescent} the morphism $\bar{X}\times_k\alpha_{p,k}\to\bar{X}$ descends to $X\times_k\alpha_{p,k}\to X$, a contradiction!
\end{proof}

\section{The Second Fundamental  Sequence}

\begin{thm} \label{exactness2} Let $X$ be a geometrically connected separable scheme  over a field $k$, and $\bar{x}:\Spec(\bar{k})\hookrightarrow X$ be a geometric point. Then there is a
 natural sequence of $\bar{k}$-group schemes \begin{equation}\label{seq2}1\to\pi^{I}(\bar{X}/\bar{k},\bar{x})\to \pi^{I}(X/k,\bar{x})\times_k\bar{k}\to  \pi^{I}(k/k,\bar{x})\times_k\bar{k}\to 1.\tag{2}\end{equation} It is a complex, always exact on the right, exact on the left if $k$ is perfect and if $X$ is quasi-compact and quasi-separated, but it is in general not  exact in the middle for $I=N,E,L$.
\end{thm}
\begin{proof} The homomorphism $\theta:\pi^{I}(\bar{X}/\bar{k},\bar{x})\to \pi^{I}(X/k,\bar{x})\times_k\bar{k}$ is obtained via composing $\chi_{\bar{k}/k}^I$ (cf. \ref{comparison0}) with the canonical morphism $$\delta:\hspace{10pt}\pi^I(\bar{X}/k,\bar{x})\times_k\bar{k}\to\pi^I({X}/k,\bar{x})\times_k\bar{k}$$ obtained by base-change from the morphism $\pi^I(\bar{X}/k,\bar{x}) \to\pi^I({X}/k,\bar{x})$ in the first fundamental sequence (\ref{seq1}). The fact that (\ref{seq2}) is a complex and that the right map is surjective follows from \ref{exactness}. As the image of $\theta$ is contained in the image of $\delta$ the failure of  exactness of (\ref{seq2}) follows from that of (\ref{seq1}). Now we show the left injectivity assuming $k$ perfect and $X$ q.c. and q.s..

Since $X$ is q.c. and q.s. and $k$ is perfect, any saturated triple $(P,G,p)\in I({\bar{X}}/\bar{k},\bar{x})$ is defined over some $X\times_kl$, where $l$ is a finite separable extension of $k$.  The Weil restriction $\Res_{X\times_kl/X}(P)$ is then a torsor under $\Res_{l/k}(G)$ over $X$ \cite[7.6, Theorem 4 and Proposition 5]{BLR}, and there are canonical adjunction maps $$\phi:\Res_{X\times_kl/X}(P)\times_k\bar{k}\to P\hspace{20pt}\text{and}\hspace{20pt}h:\Res_{l/k}(G)\times_k\bar{k}\to G$$ where $\phi$ has to be surjective for $P$ is a connected scheme. By choosing a $\bar{k}$-point $q$ in the fibre of $p\in P(\bar{k})$ we get a morphism $$(\phi,h): \ (\Res_{X\times_kl/X}(P)\times_k\bar{k}, \Res_{l/k}(G)\times_k\bar{k},q)\to (P,G,p)\in I({\bar{X}}/\bar{k},\bar{x}).$$ This means that for any surjective $\bar{k}$-homomorphism $\pi^I(\bar{X}/\bar{k},\bar{x})\twoheadrightarrow G$ we can find a ${k}$-group scheme $H$ (namely $ \Res_{l/k}(G)$) and a homomorphism $\pi^I(X/k)\to H$ with a commutative diagram 
$$\xymatrix{\pi^I(\bar{X}/\bar{k},\bar{x})\ar@{>>}[rr]\ar[dr]^-{\phi}&&G\\&H\times_k\bar{k}\ar@{>>}[ur]&}$$ where $\phi$ is the natural composition $\pi^I(\bar{X}/\bar{k},\bar{x})\xrightarrow{\theta}\pi^I(X/k)\times_k\bar{k}\to H\times_k\bar{k}$. Thus $\theta$ induces a surjection of the Hopf-algebras, and is therefore a closed imbedding.
\end{proof}

\section{Further Remarks}

Since Nori's original definition is very geometric, it is hard to adapt some arithmatic problems arised from the \'etale fundamental group to Nori's setting. One of such  arithmetic problems is the section conjecture:

Let $X$ be a smooth projective geometrically connected curve of genus $\geq 2$ over a field $k$ finitely generated over $\Q$. Let $\bar{x}$ be a geometric point of $X$. Consider the fundamental exact sequence. \begin{equation}1\to \pi_1^{\text{\rm\'et}}(\bar{X},\bar{x})\to\pi_1^{\text{\rm\'et}}({X},\bar{x})\xrightarrow{\pi} \Gal(\bar{k}/k)\to 1\tag{FES}\end{equation} Let $\sect_{}(k,X)$ be the set of sections of {$\pi$}, i.e. continous group homomorphisms from $\Gal(\bar{k}/k)\to\pi_1^{\text{\rm\'et}}({X},\bar{x})$ whose composition with $\pi$ is the identity of $\Gal(\bar{k}/k)$. In $\sect_{}(k,X)$ we define an equivalence relation: Two sections $f,g$ are equivalent if there exists an element $a\in\pi_1^{\text{\rm\'et}}(\bar{X},\bar{x})$ such that $f$ and $g$ differ by the inner automorphism of $\pi_1^{\text{\rm\'et}}({X},\bar{x})$ defined by $a$. We denote $\sect_{\sim}(k,X)$ the set of \textit{sections classes}. If $X$ has a rational point $y\in X(k)$, then we get a section class $y_*\in\sect_{\sim}(k,X)$ by the functoriality of $\pi_1^{\text{\rm\'et}}$. It can be shown that the so defined map $X(k)\to \sect_{\sim}(k,X)$ is injective. The section classes of the form $y_*$ are called \textit{geometric sections}.

\begin{conj}\label{section conjecture} (Grothendieck's section conjecture) 
All sections in $\sect_{\sim}(k,X)$ are geometric sections.
\end{conj}

Since in Nori's original definition the fundamental group scheme of a field is trivial, it is not possible to directly reformulate the section conjecture  in this setting. In \cite{EHai} and \cite{EHai2}, H. Esnault and P. H. Hai successfully used the language of fundamental groupoid scheme to get some arithmetic information from Nori's geometric fundamental group scheme, and using this they reformulated the section conjecture and proved the Packet conjecture. In \cite{BV}, N. Borne and A. Vistoli greatly generalized Nori's definition, and using the language of gerbes they also gave a reformulation of this conjecture. Here we would  like to suggest another thought on this conjecture.

\begin{lem} Let $X$ be a connected reduced scheme over a field $k$, and $\bar{x}\in X(\bar{k})$. Then we have a canonical map of sets $$\Delta:\sect^N(k,X)\longrightarrow\sect(k,X)$$ where $\sect^N(k,X)$ denotes the set of sections of the canonical surjection $\pi^N(X/k,\bar{x})\twoheadrightarrow\pi^N(k/k,\bar{x})$.
\end{lem}
\begin{proof} Suppose $f\in\sect^N(k,X)$ is a section.  Consider the following commutative diagram $$\xymatrix{\pi^N(k/k,\bar{x})\ar[r]^-f\ar@{>>}[d]^-{\pi_G^N}&\pi^N(X/k,\bar{x})\ar@{>>}[r]\ar@{>>}[d]^-{\pi_G^N}&\pi^N(k/k,\bar{x})\ar@{>>}[d]^-{\pi_G^N}\\\pi^G(k/k,\bar{x})\ar@{.>}[r]^-{\phi}&\pi^G(X/k,\bar{x})\ar@{>>}[r]&\pi^G(k/k,\bar{x})}$$
where $\phi$ is obtained by the universality of $\pi^N(k/k,\bar{x})\xrightarrow{\pi_G^N} \pi^G(k/k,\bar{x})$: For any homomorphism $\lambda:\pi^N(k/k,\bar{x})\to M$ where $M$ is a  pro-constant group scheme, there is a unique homomorphism $\delta:\pi^G(k/k,\bar{x})\to M$ such that $\delta\circ\pi_G^N=\lambda$. Clearly $\phi\in\sect(k,X)$. 
\end{proof}

\label{Nori section}Now denote $\sect_{\sim}^N(k,X)$ the image of the following composition. $$\sect^N(k,X)\xrightarrow{\Delta}\sect(k,X)\to \sect_{\sim}(k,X)$$
Then $\sect_{\sim}^N(k,X)\subseteq \sect_{\sim}(k,X)$ becomes a subset.

\begin{lem} The subset $\sect_{\sim}^N(k,X)$ contains all the geometric sections. 
\end{lem}
\begin{proof} Let $y\in X(k)$ be a rational point, and $\bar{y}\in X(\bar{k})$ be the composition of $y$ with $\Spec(\bar{k})\to\Spec(k)$. Then we have two fibre functors $\bar{F}_{\bar{x}}, \bar{F}_{\bar{y}}$ from the category of finite \'etale covers ${\rm Ecov}(\bar{X})$ to the category of sets and also the following  diagram of categories.
$$$$
$$
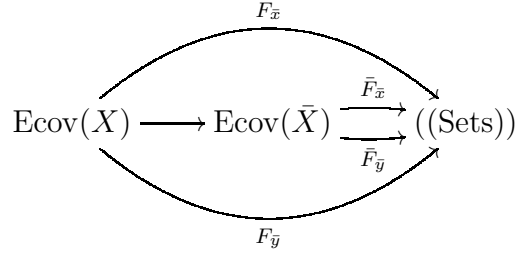
\xymatrix{ {\rm Ecov}({X})\ar[r]\ar@/^3pc/[rr]^-{F_{\bar{x}}}\ar@/_3pc/[rr]_-{F_{\bar{y}}}& {\rm Ecov}(\bar{X})\ar@/^/[r]^-{\bar{F}_{\bar{x}}}\ar@/_/[r]_-{\bar{F}_{\bar{y}}}& (({\rm Sets}))}$$
$$$$
Now fix an isomorphism $\bar{F}_{\bar{x}}\xrightarrow{\cong} \bar{F}_{\bar{y}}$, it will then induce an isomorphism $\phi:{F}_{\bar{x}}\xrightarrow{\cong} {F}_{\bar{y}}$. Going through the proof of \ref{basepoint}, we get an isomorphism $\pi^N(X/k,\bar{x})\xrightarrow{\cong_{\phi}}\pi^N(X/k,\bar{y})$ which fits into the following commutative diagram. $$\xymatrix{\pi^N(X/k,\bar{x})\ar@{>>}[d]\ar[rr]^-{\cong_{\phi}}&&\pi^N(X/k,\bar{y})\ar@{>>}[d]\\\pi^G(X/k,\bar{x})\ar[rr]^-{\cong_{\phi}}&&\pi^G(X/k,\bar{y})}$$ This implies immediately that the geometric section $y_*$ comes from $\sect_{\sim}^N(k,X)$. 
\end{proof}

Thus $\sect_{\sim}^N(k,X)$ is a possibly smaller subset of $\sect_{\sim}(k,X)$,  and the section conjecture would immediately imply:

\begin{conj}\label{reform conj} If $X$ is a proper smooth and geometrically connected curve of genus $\geq 2$ over a  field $k$ finitely generated over $\Q$, then all sections in $\sect_{\sim}^N(k,X)$ are geometric sections.
\end{conj}

\begin{rmk}\label{char p} The above conjecture could also be formulated in the case when $k$ is a global field. If $k$ is of characteristic $p>0$, there is a little subtlety in the original formulation of the section conjecture, as there might be closed points in $X$ whose residue field are non-trivial purely inseparable extensions of $k$, and such points also contribute to sections of the fundamental exact sequence (because the Galois group is insensitive to purely inseparable extenisons). Thus in positive characteristic one may expect a smaller subset of sections which correspond to the rational points. In this situation, $\sect_{\sim}^N(k,X)$ might do this job as the Nori-Galois group does sensitive to purely inseparable extensions. There is another formualtion by F. Pop (see \href{https://www.math.upenn.edu/~pop/Research/files-Res/AnabPhen_20Dec10.pdf}{Anabelian Phenomena}, pp. 32), where he extended the set of rational points to all closed points whose residues fields are purely inseparable.
\end{rmk}

There is a more general philosophy behind this section conjecture, namely the anabelian conjecture:
\begin{conj} (Grothendieck's anabelian conjecture) Let $k$ be a field finitely generated over $\Q$. Let $X,Y$ be two (proper) anabelian schemes over $k$, and $\bar{x},\bar{y}$ are geometric points of $X$ and $Y$.  Then the natural map $${\Hom}_{\Sch/k}(X,Y)\to{\Hom}_{\Gal(k)}(\pi_1^{\text{\rm\'et}}(X,\bar{x}),\pi_1^{\text{\rm\'et}}(Y,\bar{y}))/{\rm Inn}(\pi_1^{\text{\rm\'et}}(\bar{Y},\bar{y}))$$
is bijective, where $\pi_1^{\text{\rm\'et}}(X,\bar{x})$ and $\pi_1^{\text{\rm\'et}}(Y,\bar{y})$ are viewed as groups over $\Gal(k)$ and the quotient is with respect to the action of $\pi_1^{\text{\rm\'et}}(\bar{Y},\bar{y})$ on the target via inner automorphisms.
\end{conj}

 Roughly speaking the conjecture predicts that there is a full subcategory of $k$-schemes, i.e. the conjectural \textit{anabelian scheme}, which are reconstrutible from their \'etale fundamental groups. If $X$ is taken to be the base field, then this is more or less just the section conjecture.
I my opinion, Nori's fundamental group scheme carries more information than the \'etale fundamental group, thus a scheme should be more reconstructible from its fundamental group scheme. To start with,  M. Romagny, G. Zalamansky and me,  we formulated the following conjecture which is an analog of the Neukirch-Uchida theorem (\cite[12.2.1, pp. 792]{NSW}) in the purely inseparable settings.

\begin{conj} Let $k=\bar{k}$ be an algebraically closed field.  Let $K/k$ be a field extension, and $\PI(K)$ be the category of finite purely inseparable extensions of $K$ whose morphisms are just $K$-algebra homomorphisms. Let $\text{Gr.Sch}_k/\pi^L(K/k)$ be the category of $k$-group schemes over the fixed $k$-group scheme $\pi^L{(K/k)}$, i.e. the category of $k$-homomorphisms from some $k$-group scheme to $\pi^L{(K/k)}$. Then the canonical contravariant functor
$$F: {\PI(K)} \longrightarrow {\text{Gr.Sch}}_k/{\pi^L(K/k)}$$
$$L/K \mapsto \pi^L(L/k)/\pi^L(K/k)$$
is fully faithful.
\end{conj}

Note that since here we   only consider  torsors under finite infinitesimal group schemes, by \ref{basepoint2}, the  fundamental group schemes are \textit{canonically} isomorphic when we choose different base points. Thus we don't need the quotient by the inner automorphisms to erase the effect brought by the choice of the base points.


\begin{thebibliography}{9999999}

\bibitem[AV]{AV} Gerard van der Geer, B. Moonen, \textit{Abelian Varieties}, preliminary version of the first chapters.
\bibitem[BLR]{BLR} S. Bosch, W.L\"utkebohmert, M.Raynaud, \textit{N\'eron Models}, Springer-Verlag Berlin Heidelberg, 1990.
\bibitem[BV]{BV} N. Borne, A. Vistoli, \textit{The Nori fundamental gerbe of a fibered category}, J. Algebraic Geometry, S 1056-3911, 00638-X, 2014.
\bibitem[Del]{Del} P. Deligne, \textit{Cat\'egories Tannakiannes}, The Grothendieck Festschrift, Vol. II, Progr. Math., vol. 87, Birkh\"auser Boston, Boston, MA, pp. 111-195, 1990.
\bibitem[DM]{DM} P. Deligne and J. S .Milne, \textit{Tannakian categories}, in \textit{Hodge Cycles, Motives, and Shimura Varieties} by P.Deligne, J.S.Milne, A.Ogus, K.Shih, Lecture Notes in Math. 900, Springer-Verlag, 1982, pp.414.
\bibitem[EH]{EH} H. Esnault, A. Hogadi, \textit{On the algebraic
fundamental group of smooth varieties in characteristic $p>
0$}, Trans. Amer. Math. Soc. 364, no. 5, 2429-2442, 2012.
\bibitem[EHai]{EHai} H. Esnault, P. H. Hai, \textit{The fundamental groupoid scheme and applications}, Annales de L'institut Fourier, Tome 58, $\text{n}^o$ 7, pp. 2381-2412, 2008.
\bibitem[EHai2]{EHai2} H. Esnault, P. H. Hai, \textit{Packets in Grothendieck's Section Conjecture}, Advances in Mathematics, 218,   pp. 395-416, Elsevier,  2008.
\bibitem[EHV]{EHV} H. Esnault, P. H. Hai,  E. Viehweg, \textit{On the
homotopy exact sequence for Nori's fundamental group}, \url{http://arxiv.org/abs/0908.0498}, 2010.
\bibitem[EGA]{EGA IV-1} A. Grothendieck, EGA IV: \textit{\'Etude locale des sch\'ema et des morphismes de sch\'emas}, Premi\`ere partie, Publications math\'ematiques de l'I.H.\'E.S., 20, 5-259, 1964.
\bibitem[EGA]{EGA IV-4} A. Grothendieck, EGA IV: \textit{ \'Etude locale des sch\'ema et des morphismes de sch\'emas}, Quatri\`eme partie, Publications math\'ematiques de l'I.H.\'E.S., 32, 5-361, 1967.
\bibitem[EPS]{EPS}  H. Esnault, P. H. Hai, X. Sun, \textit{On Nori's Fundamental Group Scheme}, Progress in Mathematics,
Vol.265, 366-398, Basel/Switzerland:  Birkh\"{a}user Verlag, 2007.
\bibitem[MS]{MS}V. B. Mehta, S. Subramanian, \textit{On the fundamental group scheme}, Invent. math. 148 , 143-150, 2002.
\bibitem[Nori]{Nori}Madhav V. Nori, \textit{The fundamental group schemes},
Proc.Indian Aacd.Sci. 91, 73-122, 1982.
\bibitem[Nori2]{Nori2}Madhav V. Nori, \textit{The fundamental group-scheme of an abelian variety},
Math.Ann. 263, 263-266, 1983.
\bibitem[NSW]{NSW} J. Neukirch, A. Schmidt, K. Wingberg, \textit{Cohomology of number fields}, Second Edition, Springer-Verlag, 2013.
\bibitem [Ro]{Ro} M. Romangny,  \textit{Le groupe fondamental du point},  \url{https://perso.univ-rennes1.fr/matthieu.romagny/exposes/groupe_fondamental_du_point.pdf}, 2014.
\bibitem[SGA1]{SGA1} A. Grothendieck, M. Raynaud, \textit{Rev\^{e}tements \'Etales
et Groupe Fondamental}, SGA 1, Springer-Verlag, 1971.
\bibitem[SGA3]{SGA3} M. Demazure, A. Grothendieck,  \textit{Sch\'ema en groupes}, Lecture notes in Mathematics, 151, 152, 153, Springer-Verlag, Berlin, 1970.
\bibitem[SGA4]{SGA4} M. Artin, A. Grothendieck, J-L. Verdier, \textit{Th\'eorie de Topos et Cohomologie Etale des Sch\'emas}, SGA 4, Springer-Verlag, 1963/64.
\bibitem[Sz]{Sz} T. Szamuely, \textit{Galois Groups and Fundamental Groups}, Cambridge Studies in Advanced Mathematics, Vol. 117, 2009.
\bibitem[Zh]{Zh} L.  Zhang, \textit{The Homotopy Sequence of Nori's Fundamental Group}, Journal of Algebra, Vol. 393, pp. 79–91, 2013.
\bibitem[Zh2]{Zh2} L.Zhang, \textit{The homotopy sequence of the algebraic fundamental group}, IMRN, doi: 10.1093/imrn/rnt163, 2013.
\end{thebibliography}
\end{document}